\documentclass[11pt]{article}
\usepackage{pifont}
\usepackage{amsfonts}
\usepackage{amscd}
\usepackage{amsmath,amstext,amsthm,amsbsy,amssymb}
\usepackage{mathrsfs}

\newcommand{\bc}{{\mathbb C}}
\newcommand{\br}{{\mathbb R}}
\newcommand{\bh}{{\mathbb H}}
\newcommand{\ba}{{\mathbb A}}
\newcommand{\bs}{{\mathbb S}}

\newcommand{\bm}{{\mathbb M}}
\newcommand{\F}{{\mathbb F}}
\newcommand{\bp}{{\mathbb P}}
\newcommand{\fp}{{\mathfrak{p}}}
\newcommand{\fq}{{\mathfrak{q}}}
\newcommand{\p}{{\bf p}}
\newcommand{\n}{{\bf n}}
\newcommand{\q}{{\bf q}}
 \newcommand{\X}{{\mathbb X}}
\newcommand{\mtx}[4]

\newtheorem{thm}{Theorem}[section]
\newtheorem{lem}{Lemma}[section]
\newtheorem{prop}{Proposition}[section]

\newtheorem{dfn}{Definition}[section]

\begin{document}

\title{ Congruence classes of points in  quaternionic  hyperbolic space}
\author{Wensheng Cao  \\
School of Mathematics and Computational Science,\\
 Wuyi University,
Jiangmen, \\
Guangdong 529020, P.R. China\\
e-mail: {\tt wenscao@aliyun.com}\\
 }
\date{}
\maketitle

\bigskip
{\bf Abstract}\ \  An important problem in  quaternionic  hyperbolic geometry is to classify ordered $m$-tuples of pairwise distinct points in the closure of quaternionic hyperbolic n-space,  $\overline{{\bf H}_\bh^n}$,
 up to congruence in the holomorphic isometry group ${\rm PSp}(n,1)$ of  ${\bf H}_\bh^n$. In this paper we concentrate  on two cases:  $m=3$ in $\overline{{\bf H}_\bh^n}$ and $m=4$ on $\partial{\bf H}_\bh^n$ for $n\geq 2$.
  New  geometric  invariants  and several distance formulas in  quaternionic hyperbolic geometry  are introduced and studied for this problem.
 The  congruence classes are completely described  by  quaternionic Cartan's angular  invariants and the distances between some geometric objects for the first case.  The moduli space  is constructed for the second case.
\bigskip

\medskip


\section{Introduction}

    Let $\F$ denote the real numbers $\br$, the complex numbers $\bc$ or the quaternions $\bh$.
Let ${\bf H}_\F^n$ denote the $n$-dimensional hyperbolic space over $\F$ with the boundary  $\partial{\bf H}_\F^n$ .  Let $\langle ,\rangle$ be  the Hermitian product in $\F^{n+1}$ of signature $(n,1)$. For $\F=\bc$ and $\bh$,
 the linear groups which act as the isometries in ${\bf H}_\F^n$, are  denoted by ${\rm PU}(n,1)$ and ${\rm PSp}(n,1)$ respectively.
An important problem in   hyperbolic geometry is to classify ordered $m$-tuples of pairwise distinct points  in the closure of hyperbolic n-space  $\overline{{\bf H}_\F^n}$
 up to congruence in the holomorphic isometry group of ${\bf H}_\F^n$.   This
problem is trivial for $m = 1, 2$. To deal with the cases of  $m\geq 3$,  one need to develop  some geometric invariants or geometric tools.

   The  cross-ratio  of a quadruple of  points in $\partial{\bf H}_\br^{n}$  was  defined by Cao and Waterman  \cite{caowat98},  which  coincides with the  classical cross-ratio of the  complex plane when $n=3$  and  the quaternionic cross-ratio of a quadruple of  quaternions \cite{bisgen09} when $n=5$. These classical cross ratios   are  useful tools in  real hyperbolic geometry.

Let $\fp=(p_1,p_2,p_3,p_4)$ be an ordered quadruple of pairwise distinct points in $\partial{\bf H}_\bc^{n}$.  The classical cross-ratio was generalized to $\partial{\bf H}_\bc^{n}$  by Kor\'anyi and Reimann \cite{korrei87} as the following complex number:
\begin{equation}\label{complexcross}
\X(\fp)=\X(p_1,p_2,p_3,p_4)=\frac{\langle \p_3,\p_1 \rangle \langle \p_4, \p_2 \rangle}{\langle \p_4,\p_1 \rangle \langle \p_3, \p_2 \rangle},
\end{equation}
where  $\p_i\in \bc^{n,1}$ are   null lifts of $p_i$.
  This complex cross-ratio is  closely related to  Cartan's angular  invariant  and other geometric invariants.  The Cartan's angular  invariant \cite{car32,gol99} is an angle associated to a triple $\fp=(p_1,p_2,p_3)$ of points in  $\partial{\bf H}_\bc^n$.  Such an angle  $\ba(\fp)$  is defined to be the following argument:
    \begin{equation}\label{complexangule}
\ba(\fp)=\arg(-\langle \p_1,\p_2,\p_3 \rangle)\in [-\pi/2, \pi/2],
\end{equation}
where $\p_i$ are the null lifts of $p_i$, and
$$\langle \p_1,\p_2,\p_3 \rangle= \langle \p_1,\p_2 \rangle \langle \p_2,\p_3 \rangle \langle \p_3,\p_1 \rangle.$$
   It is a shape invariant, originally  used in detecting whether the corresponding  triple lies on a chain or on an $\br$-circle.

The  Cartan's angular invariant $\ba(\fp)$ and some distance formulas between some geometric objects are the exact geometric invariant and tools  to  study  congruence classes of triples  $\fp=(p_1,p_2,p_3)\in (\overline{{\bf H}_{\bc}^n})^3$.  The  Cartan's angular invariant $\ba(\fp)$  of a tripe  $\fp=(p_1,p_2,p_3)\in (\partial{\bf H}_\bc^n)^3$  determines  its congruence class in ${\rm PU}(n,1)$ \cite{car32,gol99}.
 The moduli space of such  triples can be described as the interval $[-\pi/2,\pi/2]$.  If  such a triple $\fp=(p_1,p_2,p_3)$  is in  ${\bf H}_{\bc}^n$, then its congruence class in ${\rm PU}(n,1)$  is described by the three distance $\rho(p_i,p_j)$ and Brehm's shape invariant \cite{bre90}. Here $\rho(,)$ is the Bergman metric on ${\bf H}_{\bc}^n$. The distance formula from  a  point in   ${\bf H}_{\bc}^n$ to a complex geodesic is needed for the general congruence class problem in $\overline{{\bf H}_\bc^n}$.

For $n = 2$ and $m = 4$  this problem was considered by Falbel, Parker and Platis \cite{falpla08,fal09,park18,park19}. The main tool  is the complex cross-ratio variety determined by three complex cros-ratios.  The moduli space of ordered quadruples of pairwise distinct points in $\partial{\bf H}_\bc^n$ was described by Cunha and Gusevskii \cite{cungus10} with the tool of Gram matrix.   A Gram matrix associated to  $\fp=(p_1,p_2,p_3,p_4)$  in $\partial {\bf H}_{\bh}^n$ with lift  ${\bf p}=({\bf p}_{1},{\bf p}_{2},{\bf p}_{3}, {\bf p}_{4})$ is the Hermitian  matrix
$$G=G({\bf p})=(g_{ij})=( \langle {\bf p}_{i},{\bf p}_{j}\rangle).$$
Gram matrix  is an important tool in  complex hyperbolic geometry \cite{cungus10,gol99,gro06}   because  its entries $\langle {\bf p}_{i},{\bf p}_{j}\rangle$   are  base material for other geometric invariants. We can read off it almost  all the geometric information concerning the relative geometric positions of $p_i$.

  This technique is also used to  construct the invariants which describe
uniquely the ${\rm PU}(n,1)$-congruence class of an ordered $m$-tuple of pairwise distinct points in
$\partial{\bf H}_\bc^n$  and describe the corresponding moduli
space for any $n \geq 1$ and $m \geq 4$ \cite{cungus12}.

The main aim of this paper is concerned with the important problem mention above in quaternionic hyperbolic geometry. We  concentrate ourself on two cases:  $m=3$ in $\overline{{\bf H}_\bh^n}$ and $m=4$ on $\partial{\bf H}_\bh^n$ for $n\geq 2$. We will consider  the description of  congruence classes for the first case.  We also will  obtain a construction of   the moduli space of ordered quadruples of pairwise distinct points in  $\partial{\bf H}_\bh^n$.

 For this purpose  we  need to  develop some  geometric invariants and  tools in  quaternionic hyperbolic geometry.
We will extend the quaternionic Cartan's angular  invariant and quaternionic cross-ratio to $\overline{{\bf H}_\bh^n}$.  Several distance  formulas between some geometric objects are  obtained.  We will introduce the Gram matrix in quaternionic  hyperbolic geometry.  We remark that some generalizations are new even  in complex hyperbolic geometry and the introduced  geometric invariants  deserve further research  in  quaternionic hyperbolic geometry.

  We need several notations and definitions to state our main results.  Such notations and  definitions rely on  the following two propositions, which will be proved in Section \ref{sect-pre} below.
\begin{prop}\label{prop-1.1} If ${\bf z},{\bf w}\in \bh^{n,1}-\{0\}$  with $\langle {\bf z},\,{\bf z} \rangle\leq 0 $ and  $\langle{\bf w},\,{\bf w}\rangle\leq 0 $ then either ${\bf w}={\bf z}\lambda$ for some $\lambda\in \bh$ or  $\langle{\bf z},\,{\bf w}\rangle\neq 0$.
\end{prop}

\begin{prop}\label{prop-1.2}
Let   $\fp=(p_1,p_2,p_3)$  be any triple of  pairwise distinct points in $\overline{{\bf H}_{\bh}^n}$ and  $\p_1,\p_2,\p_3$ be  arbitrary lifts of $p_1, p_2, p_3$  respectively, then the number
$$ \langle \p_1, \p_2, \p_3\rangle =\langle \p_2, \p_1 \rangle \langle \p_3, \p_2\rangle \langle \p_1, \p_3 \rangle\in \bh,
$$
and
 \begin{equation}\Re(\langle \p_1, \p_2, \p_3\rangle )\leq  0 .\end{equation}
\end{prop}

By  Propositions \ref{prop-1.1},\ref{prop-1.2}, we can rephrase  the definition of quaternionic Cartan's angular  invariant by  Apanasov and Kim \cite{apakim07} as follows.
  \begin{dfn} \label{car32 angularh} The {\it quaternionic Cartan's angular  invariant} of a triple  $\fp=(p_1,p_2,p_3)$   of pairwise distinct points in $\overline{{\bf H}_{\bh}^n}$ is the  angular invariant $\ba_{\bh}(\fp)$,  $0 \leq \ba_{\bh}(\fp)\leq \frac{\pi}{2}$,    given by
\begin{equation} \label{angular} \ba_{\bh}(\fp)=\ba_{\bh}(p_1,p_2,p_3):=\arccos \frac{\Re(-\langle \p_1, \p_2, \p_3\rangle)}{|\langle \p_1, \p_2, \p_3\rangle|},\end{equation}
where $\p_1,\p_2,\p_3$ are lifts of $p_1, p_2, p_3$, respectively.
\end{dfn}

  \begin{dfn}
 For $u,v\in \overline{{\bf H}_{\bh}^n}$ with  lifts ${\bf u}$ and ${\bf v}$, respectively,   we define the  {\it quaternionic line} spanned by $u,v$ as the set
 \begin{equation}
  L_{uv}=\bp(\{ {\bf w}:{\bf w}={\bf u}\lambda+{\bf v}\mu,  \lambda, \mu\in \bh\})\cap \overline{
{\bf H}_{\bh}^n}.
\end{equation}
\end{dfn}
Let $\rho(,)$ be the Bergman metric on ${\bf H}_{\bh}^n$  and  $\rho(L_{uv},z)$  the  hyperbolic  distance  from $z\in  {\bf H}_{\bh}^n$ to  $L_{uv}$.

\medskip
One of our main results is the following theorem.
\begin{thm}\label{thm-triple}  Let  $\fp=(p_1, p_2, p_3)$ and  $\fq=(q_1, q_2, q_3)$ be
triples of  pairwise distinct points in   $\overline{{\bf H}_{\bh}^n}$.   Then there exists an isometry
$h\in  {\rm Sp}(n,1)$ such that $h(p_1)=q_1, h(p_2)=q_2,h(p_3)=q_3$  if and only if one of the following conditions holds:
\begin{itemize}
  \item[(i)] $\fp,\fq \in \partial {\bf H}_{\bh}^n\times \partial {\bf H}_{\bh}^n\times \partial {\bf H}_{\bh}^n$ and $\ba_{\bh}(\fp)= \ba_{\bh}(\fq)$.
  \item[(ii)]  $\fp,\fq \in \partial {\bf H}_{\bh}^n\times \partial {\bf H}_{\bh}^n\times {\bf H}_{\bh}^n$, $\ba_{\bh}(\fp)= \ba_{\bh}(\fq)$ and $\rho(L_{p_1p_2},p_3)=\rho(L_{q_1q_2},q_3)$.
  \item[(iii)]  $\fp,\fq \in \partial {\bf H}_{\bh}^n\times {\bf H}_{\bh}^n\times {\bf H}_{\bh}^n$, $\rho(L_{p_1p_2},p_3)=\rho(L_{q_1q_2},q_3)$, $\rho(L_{p_1p_3},p_2)=\rho(L_{q_1q_3},q_2)$ and  $\rho(p_2,p_3)=\rho(q_2,q_3).$
  \item[(iv)]  $\fp,\fq \in {\bf H}_{\bh}^n\times {\bf H}_{\bh}^n\times {\bf H}_{\bh}^n$, $\ba_{\bh}(\fp)= \ba_{\bh}(\fq)$, $\rho(p_1,p_2)=\rho(q_1,q_2)$,  $\rho(p_1,p_3)=\rho(q_1,q_3)$ and $\rho(p_2,p_3)=\rho(q_2,q_3).$
\end{itemize}
\end{thm}

For two quaternions  $a=a_0+a_1{\bf i}+a_2{\bf j}+a_3{\bf k}$ and $b=b_0+b_1{\bf i}+b_2{\bf j}+b_3{\bf k}$,  where  $a_i,b_i\in\br$. We define  the following two functions:
$$\nu:\bh\to \bh, \ \sigma:\bh\times \bh\to \bc$$ as
 \begin{equation}\nu(a)=\left\{
             \begin{array}{ll}
               \frac{(\sqrt{a_1^2+a_2^2+a_3^2}+a_1)-a_3{\bf j}+a_2{\bf k}}{\sqrt{2(a_1^2+a_2^2+a_3^2)+2a_1\sqrt{a_1^2+a_2^2+a_3^2}}}, & \hbox{provided}\ a_2^2+a_3^2\neq 0; \\
               {\bf j}, & \hbox{provided}\ a_1<0; \\
              1, & \hbox{otherwise}
             \end{array}
           \right.\end{equation}
           and
           \begin{equation}\sigma(a,b)=\left\{
             \begin{array}{ll}
               \sqrt{\frac{a_2+a_3{\bf i}}{\sqrt{a_2^2+a_3^2}}}, & \hbox{provided}\ a_2^2+a_3^2\neq 0; \\
               \sqrt{\frac{b_2+b_3{\bf i}}{\sqrt{b_2^2+b_3^2}}}, & \hbox{provided}\ b_2^2+b_3^2\neq 0; \\
              1, & \hbox{otherwise.}
             \end{array}
           \right.\end{equation}
       We mention that $|\nu(a)|=|\sigma(a,b)|=1$  and $\nu(a)$ and  $\sigma(a,b)$  can be viewed as orthogonal rotations such that
 $$\nu(a)^{-1} a \ \nu(a)=a_0+\sqrt{a_1^2+a_2^2+a_3^2}\ {\bf i}$$
 and
 $$\sigma(a,b)^{-1} a  \sigma(a,b)=a_0+a_1{\bf i}+\sqrt{a_2^2+a_3^2}\ {\bf j},\ \mbox{provided}\  a_2^2+a_3^2\neq 0$$
 or
$$\sigma(a,b)^{-1} b  \sigma(a,b)=b_0+b_1{\bf i}+\sqrt{b_2^2+b_3^2}\ {\bf j},\ \mbox{provided}\  a_2^2+a_3^2=0,  b_2^2+b_3^2\neq 0.$$

Let $\mathcal{M}(n)$ be the configuration space of  quadruples of pairwise distinct points in $\partial {\bf H}_{\bh}^n$, that is, the quotient of the set of quadruples of pairwise distinct points in $\partial {\bf H}_{\bh}^n$  with respect to the diagonal action of ${\rm PSp}(n,1)$ equipped with the quotient topology.

Let $m(\mathfrak{p})\in \mathcal{M}(n)$ be the point represented by $\mathfrak{p}=(p_1,p_2,p_3,p_4)$.  We  define the map
\begin{equation}
\tau:\mathcal{M}(n)\to \bc^3\times \br\times \br
\end{equation}
by the formula
\begin{equation}\label{dfntau}
\tau: m(\fp)\longmapsto (c_1,c_2,c_3,t;\ba),
 \end{equation}
 where $(c_1,c_2,c_3,t;\ba)$ are determined by the following three steps:

{\it Step 1 : Determine $\lambda_i,i=1,2,3,4.$}

 Let ${\bf p}=({\bf p}_{1},{\bf p}_{2},{\bf p}_{3}, {\bf p}_{4})$ be an arbitrary lift of $\fp$.  It follows from Proposition \ref{prop-1.1} that $\langle {\bf p}_{i}, {\bf p}_{j}\rangle\neq 0$ for $i\neq j$.

Set   \begin{equation}\label{flam2-4}\lambda_2=\langle {\bf p}_{2}, {\bf p}_{1}\rangle^{-1},\ \lambda_3=\langle {\bf p}_{3}, {\bf p}_{2}\rangle^{-1}\langle {\bf p}_{1}, {\bf p}_{2}\rangle, \ \lambda_4=\langle {\bf p}_{4}, {\bf p}_{3}\rangle^{-1}\langle {\bf p}_{2}, {\bf p}_{3}\rangle \langle {\bf p}_{2}, {\bf p}_{1}\rangle^{-1} \end{equation}
 and
 \begin{equation}\label{flam1}\lambda_1=\frac{\nu( \langle {\bf p}_{1}, {\bf p}_{3}\lambda_3\rangle)}{\sqrt{| \langle {\bf p}_{1}, {\bf p}_{3}\lambda_3\rangle |}}=\frac{\nu(\langle {\bf p}_{2}, {\bf p}_{1}\rangle \langle {\bf p}_{2}, {\bf p}_{3}\rangle^{-1} \langle {\bf p}_{1}, {\bf p}_{3}\rangle)}{\sqrt{|\langle {\bf p}_{2}, {\bf p}_{1}\rangle \langle {\bf p}_{2}, {\bf p}_{3}\rangle^{-1} \langle {\bf p}_{1}, {\bf p}_{3}\rangle|}}.\end{equation}

 {\it Step 2 : Determine $\mu$.}

Set
 \begin{equation}\label{fmu} \mu= \sigma(\langle {\bf p}_{1}\lambda_1, {\bf p}_{4}\lambda_4 \bar{\lambda_1}^{-1} \rangle,  \langle {\bf p}_{2}\lambda_2 \bar{\lambda_1}^{-1}, {\bf p}_{4}\lambda_4 \bar{\lambda_1}^{-1}\rangle).\end{equation}

  {\it Step 3 : Determine $(c_1,c_2,c_3,t;\ba)$.}

  Set \begin{equation}\label{findrep} c_1+t{\bf j}=\bar{\mu}\langle {\bf p}_{1}\lambda_1, {\bf p}_{4}\lambda_4 \bar{\lambda_1}^{-1} \rangle \mu,\ c_2+c_3{\bf j}=\bar{\mu} \langle {\bf p}_{2}\lambda_2 \bar{\lambda_1}^{-1}, {\bf p}_{4}\lambda_4 \bar{\lambda_1}^{-1}\rangle \mu. \end{equation}

\medskip

  We will show in Section \ref{sect-moduli} that the parameters $(c_1,c_2,c_3,t;\ba)$ are independent of lifts $\p_i, i=1,2,3,4$ and are determined only by $p_i, i=1,2,3,4$.
\begin{dfn}
Let $\bm(2)$ be  the set of points $ (c_1,c_2,c_3,t;\ba)$ in $\bc^3\times \br\times \br$ defined by
\begin{equation}\label{condtmain-2} D(G)= 1+|c_1|^2+|c_2|^2+|c_3|^2+t^2-2\Re(c_1)+2\Re(c_2e^{-{\bf i}\ba})+2\Re\big((\bar{c_1}c_2+t\bar{c_3})e^{{\bf i}\ba}\big)=0\end{equation}
subject to the following restrictions:
\begin{equation}\label{restrictions-2}\ba\in[0,\pi/2], \ \  \Re(c_1\bar{c_2})+t\Re(c_3)\leq 0,\ \Re(c_2)\leq 0,\ \ t\geq 0,\ \ |c_1|^2+t^2\neq 0,\ \ |c_2|^2+|c_3|^2\neq 0.\end{equation}
Let $\bm(n) (n>2)$ be the set of points $ (c_1,c_2,c_3,t;\ba)$ in $\bc^3\times \br\times \br$ defined by
\begin{equation}\label{condtmain-n} D(G)= 1+|c_1|^2+|c_2|^2+|c_3|^2+t^2-2\Re(c_1)+2\Re(c_2e^{-{\bf i}\ba})+2\Re\big((\bar{c_1}c_2+t\bar{c_3})e^{{\bf i}\ba}\big)\leq 0\end{equation}
subject to the following restrictions:
\begin{equation}\label{restrictions-n}\ba\in[0,\pi/2], \ \  \Re(c_1\bar{c_2})+t\Re(c_3)\leq 0,\ \Re(c_2)\leq 0,\ \ t\geq 0,\ \ |c_1|^2+t^2\neq 0,\ \ |c_2|^2+|c_3|^2\neq 0.\end{equation}
\end{dfn}
We call $\bm(n)$ the moduli space for $\mathcal{M}(n)$. $\bm(n)$ is equipped with the topology induced from $\bc^3\times \br\times \br$.

\medskip

The other main result in this paper  is the following  theorem.

\begin{thm}\label{thm-moduli}
The configuration space  $\mathcal{M}(n)$   is homeomorphic to $\bm(n)$.
\end{thm}

 The paper is organized  as follows. Section \ref{sect-pre} contains  some basic facts in  quaternionic hyperbolic geometry and the proof of  Propositions \ref{prop-1.1}, \ref{prop-1.2}.   Section \ref{sect-geoinv} is devoted to  developing  some geometric invariants and geometric tools in quaternionic hyperbolic geometry.  These are   quaternionic cross-ratio, quaternionic Cartan's angular  invariant and some distance formulas.  Sections \ref{cong-trip}  contains the proof of Theorem \ref{thm-triple}.  In Section \ref{sect-moduli},  we  introduce the   Gram matrix in quaternionic hyperbolic geometry.  Besides the proof of Theorem  \ref{thm-moduli}, we also  obtain a theorem  (Theorem \ref{thmcong}) about the congruence classes of quadruples of pairwise distinct points on $\partial{\bf H}_\bh^n$.

  \section{Preliminaries}\label{sect-pre}

 We briefly recall  some  necessary material on  quaternionic hyperbolic geometry  here  and we refer to \cite{apakim07,chegre74,kimpar03} for further details.

We recall that a  quaternion is of the form $a=a_0+a_1{\bf i}+a_2{\bf j}+a_3{\bf k}\in \bh$
where $a_i\in \br$ and $ {\bf i}^2 = {\bf j}^2 = {\bf k}^2 = {\bf
i}{\bf j}{\bf k} = -1.$ Let $\overline{a}=a_0-a_1{\bf i}-a_2{\bf
j}-a_3{\bf k}$ and $|a|= \sqrt{\overline{a}a}=\sqrt{a_0^2+a_1^2+a_2^2+a_3^2}$  be the  conjugate  and modulus of $a$, respectively.  We define $\Re(a)=(a+\overline{a})/2$ and $\Im(a)=(a-\overline{a})/2$.   For $a,b\in \bh$, we have
$$\Re(ab)=\Re(ba)=\Re(\bar{a}\bar{b})=\Re(\bar{b}\bar{a}).$$  Two quaternions $a$ and $b$ are similar, denoted by $a\sim b$,  if there exists nonzero $\lambda \in \bh$  such that $b=\lambda a
\lambda^{-1}$.  We mention that $a\sim b$  if and only if  $\Re(a)=\Re(b)$  and  $|a|=|b|$.   Let  $$\bs=\{\nu=\nu_1{\bf i}+\nu_2{\bf j}+\nu_3{\bf k}:
 \nu_1^2+\nu_2^2+\nu_3^2=1,v_i\in \br\}.$$
    Every  unit quaternion $\nu$ can  be written  as
$$\nu=\exp(\theta {\bf I}):=\cos \theta +{\bf I} \sin \theta=\cos(-\theta) +(-{\bf I}) \sin(-\theta)\,\, \, \, \ \  \mbox{
for some}\ \theta \in [0,\pi]\;\;\mbox{and}\;\;{\bf I}\in \bs.$$
 It is useful to  view $\bh$ as  $\bh=\bc\oplus  \bc {\bf j}$. Therefore any quaternion $a=a_0+a_1{\bf i}+a_2{\bf j}+a_3{\bf k}$ can be uniquely expressed as
 $$a=(a_0+a_1{\bf i})+(a_2+a_3{\bf i}){\bf j}=c_1+c_2 {\bf j}=c_1+{\bf j}\bar{c_2}.$$

  Let $\bh^{n,1}$ be the  vector space   with the Hermitian form of signature $(n,1)$ given by
$$
\langle{\bf z},\,{\bf w}\rangle={\bf w}^*J{\bf z}=
\overline{w_1}z_{n+1}+\overline{w_2}z_{2}+\cdots+\overline{w_n}z_{n}+\overline{w_{n+1}}z_{1}$$ with matrix $$J=\left(
                  \begin{array}{ccc}
                    0 & 0 & 1 \\
                    0 & I_{n-1} & 0 \\
                    1 & 0 & 0\\
                  \end{array}
                \right).$$
We define  $${\rm Sp}(n,1)=\{g\in GL(n+1,\bh):A^*JA=J\}.$$
  Let   $g\in {\rm Sp}(n,1)$. Then $g$ and $g^{-1}$ are of the following forms:
\begin{equation}\label{hform}
  g=\left(
  \begin{array}{ccc}
     a& \zeta^*& b \\
    \alpha & A& \beta\\
    c & \delta^*& d\\
    \end{array}
\right), \ \  g^{-1}=\left(
  \begin{array}{ccc}
     \overline{d}& \beta^*& \overline{b} \\
    \delta & A^*& \zeta\\
    \overline{c} & \alpha^*& \overline{a}\\
    \end{array}
\right), \end{equation} where $a, b, c, d\in\bh$, $A$ is an
$(n-1)\times (n-1)$ matrix over $\bh$, and $\alpha, \beta, \zeta,
\delta$ are column vectors in $\bh^{n-1}$.

Following Section 2 of \cite{chegre74}, let
\begin{eqnarray*}
V_0 =  \Bigl\{{\bf z} \in  \bh^{n,1}\setminus\{0\}:
\langle{\bf z},\,{\bf z}\rangle=0\Bigr\},\,\,
V_{-} = \Bigl\{{\bf z} \in \bh^{n,1}:\langle{\bf z},\,{\bf
z}\rangle<0\Bigr\}.
\end{eqnarray*}
   Let
$\bp:\bh^{n,1}\setminus\{0\}\longrightarrow \bh\bp^n$ be the  right projection onto $\bh$-projective space.
 If $z_{n+1}\neq 0$ then
$\bp$ is given by
$$
\bp(z_1,\,\ldots,\,z_n,  z_{n+1})^T=(z_1z_{n+1}^{-1},\cdots,z_n
z_{n+1}^{-1})^T\in{\bh }^n.
$$
We also define
$$
\bp(z_1, 0, \,\ldots,\,0,0)^T=\infty, \ \  \bp(0, 0, \,\ldots,\,0,z_{n+1})^T=o.
$$
The Siegel domain model of the quaternionic hyperbolic $n$-space is defined to be ${\bf H}_{\bh}^n=\bp(V_-)$ with the boundary $\partial {\bf H}_{\bh}^n=\bp(V_0)$.
 We mention that  $g\in{\rm Sp}(n,1)$ acts on ${\bf H}_\bh^n\cup\partial{\bf H}_\bh^n$ as $
g(z)=\bp g\bp^{-1}(z)$. The Bergman metric on
${\bf H}_{\bh}^n$ is given by the distance formula
$$
\cosh^2\frac{\rho(z,w)}{2}=
\frac{\langle{\bf z},\,{\bf w}\rangle \langle{\bf w},\,{\bf z}\rangle}
{\langle{\bf z},\,{\bf z}\rangle \langle{\bf w},\,{\bf w}\rangle},
\ \ \mbox{where}\ \ z,w \in {\bf H}_{\bh}^n, \ \
{\bf z}\in \bp^{-1}(z), {\bf w}\in \bp^{-1}(w).
$$
The holomorphic isometry group of  ${\bf H}_{\bh}^n$ is ${\rm PSp}(n,1)={\rm Sp}(n,1)/\pm I_{n+1}$.

The   {\it standard lift}  of  a finite point $z\in \overline{{\bf H}_{\bh}^n}$ is  $$\hat {\bf z}=\left(
                   \begin{array}{c}
                     z \\
                    1 \\
                   \end{array}
                 \right)
\in  \bp^{-1}(z)$$ and  the {\it  standard lift} of $\infty$  is  \begin{equation}\hat{\bf \infty} =(-1,0,\cdots,0)^T.\end{equation}
We will follow the above convention without any other statements in the sequel.

\medskip

We close this section with the proof of Propositions \ref{prop-1.1},\ref{prop-1.2}.

\medskip

\noindent {\it Proof of Proposition \ref{prop-1.1}.}\quad Since ${\rm Sp}(n,1)$ acts transitively on  $\bh$-lines in $V_-$ and  doubly transitively on $\bh$-lines in  $V_{0}$,  and each isometry of  ${\rm Sp}(n,1)$  preserves the Hermitian form,  we  only need to  consider the following  two cases.

Firstly, we assume that  ${\bf z}=(0,\cdots,0,1)^T\in V_{0}$ and ${\bf w}=(w_1,w_2,\cdots,w_{n+1})^{T}\in V_{0}\cup V_{-}$. Then
$\langle{\bf w},\,{\bf w}\rangle=\overline{w_1}w_{n+1}+\overline{w_{n+1}}w_{1}+\sum_{i=2}^n |w_i|^2\leq 0
$ and $\langle{\bf z},\,{\bf w}\rangle=\overline{w_1}$. If $w_1=0$ then we see $w_i=0$ for $i=2,\cdots,n$ and so ${\bf w}={\bf z}w_{n+1}$.

 Secondly, we assume that  ${\bf z}=(-1,0,\cdots,0,1)^T\in V_{-}$. Then  $\langle{\bf z},\,{\bf w}\rangle=\overline{w_1}- \overline{w_{n+1}}$. Since $\langle{\bf w},\,{\bf w}\rangle\leq 0 $  we have  $\overline{w_1}\neq \overline{w_{n+1}}$, which implies that  $\langle{\bf z},\,{\bf w}\rangle\neq 0$. \hfill$\square$

\medskip

 \noindent {\it Proof of Proposition \ref{prop-1.2}.}\quad   By  properties of quaternions, we can verify that $$\langle {\p_1}\lambda_1, {\p_2}\lambda_2, {\p_3}\lambda_3\rangle=|\lambda_2\lambda_3|^2\bar{\lambda_1}\langle {\p_1}, {\p_2}, {\p_3}\rangle \lambda_1,$$  $$\Re(\langle {\p_1}, {\p_2}, {\p_3}\rangle )=\Re(\langle \p_2, \p_1 \rangle \langle \p_3, \p_2\rangle \langle \p_1, \p_3 \rangle)=\Re(\langle {\p_2}, {\p_3}, {\p_1}\rangle )=\Re(\langle {\p_3}, {\p_1}, {\p_2}\rangle )$$
  and $$ \Re(\langle {\p_1}, {\p_2}, {\p_3}\rangle )=\Re(\overline{\langle {\p_1}, {\p_2}, {\p_3}\rangle})=\Re(\langle {\p_2}, {\p_1}, {\p_3}\rangle ).$$
   These equalities imply that the sign of $\Re(\langle {\p_1}, {\p_2}, {\p_3}\rangle)$  is  independent of the choice of the lifts ${\p_1}, {\p_2}, {\p_3}$ and  invariant under permutations of the points ${\p_1}, {\p_2}, {\p_3}$.  Due to fact that the sign of $\Re(\langle {\p_1}, {\p_2}, {\p_3}\rangle)$ is invariant under the action of  isometries in  ${\rm Sp}(n,1)$, by the transitivity of ${\rm Sp}(n,1)$ on $\overline{{\bf H}_{\bh}^n}$,  we need only to consider the following two cases.

 Case (1):  One of $p_1, p_2, p_3$ lies on   $\partial {\bf H}_{\bh}^n$.

 We may assume that $p_1=\infty$ and $p_2=(u_1,\cdots,u_n)^T,\  p_3=(r_1,\cdots,r_n)^T\in \overline{{\bf H}_{\bh}^n}$.
 Note that  $$\langle {\hat{\p_1}}, {\hat{\p_2}},  {\hat{\p_3}}\rangle=r_1+\bar{u_1}+ \sum_{i=2}^{n}\bar{u_i}r_i$$
and  $-\Re(u_1)\geq \frac{1}{2}\sum_{i=2}^{n}|u_i|^2$ and  $-\Re(r_1)\geq \frac{1}{2}\sum_{i=2}^{n}|r_i|^2$.
Therefore
\begin{eqnarray*}\Re(\langle {\hat{\p_1}}, {\hat{\p_2}},  {\hat{\p_3}}\rangle)&=& \Re(u_1)+\Re(r_1)+\Re(\sum_{i=2}^{n}\bar{u_i}r_i)\\
&\leq&  \Re(u_1)+\Re(r_1)+\sqrt{\sum_{i=2}^{n}{|u_i|^2}\sum_{i=2}^{n}{|r_i|^2}}\\
&\leq& \Re(u_1)+\Re(r_1)+2\sqrt{\Re(-u_1)\Re(-r_1)}\\
&=& -\Bigr(\sqrt{-\Re(u_1)}-\sqrt{-\Re(r_1)}\Bigl)^2\leq 0.
\end{eqnarray*}

Case (2):  All three points  $p_1, p_2, p_3$  lie in ${\bf H}_{\bh}^n$.

We may assume that $p_1=(-1,0,\cdots,0)^T$ and $p_2=(u_1,\cdots,u_n)^T, p_3=(r_1,\cdots,r_n)^T$. Then
 $$-\Re(u_1)> \frac{1}{2}\sum_{i=2}^{n}|u_i|^2,\  -\Re(r_1)>\frac{1}{2}\sum_{i=2}^{n}|r_i|^2$$  and
 $$\langle {\hat{\p_1}}, {\hat{\p_2}},  {\hat{\p_3}}\rangle=(u_1-1)(r_1+\bar{u_1}+ \sum_{i=2}^{n}\bar{u_i}r_i)(\bar{r_1}-1).$$
For simplicity, let $\lambda= \sum_{i=2}^{n}\bar{u_i}r_i$.  Then we have $$|\lambda|\leq 2\sqrt{\Re(-u_1)\Re(-r_1)}$$ and
\begin{eqnarray*}&&\Re((r_1+\bar{u_1})(\bar{r_1}-1)(u_1-1))=-|r_1+\bar{u_1}|^2+(|r_1|^2+1)\Re(u_1)+(|u_1|^2+1)\Re(r_1)\\
&=&-|r_1+\bar{u_1}|^2+4\Re(u_1)\Re(r_1)+(|r_1|^2-2\Re(r_1)+1)\Re(u_1)+(|u_1|^2-2\Re(u_1)+1)\Re(r_1)\\
&=&-|r_1-u_1|^2+|\bar{r_1}-1|^2\Re(u_1)+|u_1-1|^2\Re(r_1).
\end{eqnarray*}
Therefore
\begin{eqnarray*}
\Re( \langle {\hat{\p_1}}, {\hat{\p_2}},  {\hat{\p_3}}\rangle )&=&\Re((r_1+\bar{u_1}+t)(\bar{r_1}-1)(u_1-1))\\
&=&\Re((r_1+\bar{u_1})(\bar{r_1}-1)(u_1-1))+\Re(t(\bar{r_1}-1)(u_1-1))\\
&\leq &\Re((r_1+\bar{u_1})(\bar{r_1}-1)(u_1-1))+|t||(\bar{r_1}-1)(u_1-1)|\\
&\leq &\Re((r_1+\bar{u_1})(\bar{r_1}-1)(u_1-1))+2\sqrt{\Re(-u_1)\Re(-r_1)}|(\bar{r_1}-1)(u_1-1)|\\
&=&-|r_1-u_1|^2-\Bigr(|\bar{r_1}-1|\sqrt{\Re(-u_1)}-|u_1-1|\sqrt{\Re(-r_1)}\Bigl)^2\\
&\leq &  0.
\end{eqnarray*} \hfill$\square$

\medskip

\section{ Quaternionic geometric invariants}\label{sect-geoinv}

\subsection{Quaternionic cros-ratios}

 The complex cross-ratio given by (\ref{complexcross}) is independent of the choice of $\p_i$. This property enables one to use the cross-ratios either in $\partial {\bf H}_{\bc}^n$ or $V_0$ freely. In an inner product space, the concept of cross-ratio  always stems  from its inner products.   Due to the non-commutativity of quaternions, $\overline{{\bf H}_{\bh}^n}$  and $V_0\cup V_-$  are two different worlds for the concept of quaternionic cross-ratio. In this subsection, we  define two  quaternionic cross-ratios. The first one is defined on  $V_0\cup V_-$  and the second one is defined on  $\overline{{\bf H}_{\bh}^n}$.

Let $\p=(\p_1,\p_2,\p_3,\p_4)$ be a quadruple of points in $V_0\cup V_-$, we define
$$\bp(\p)=\big(\bp(\p_1),\bp(\p_2),\bp(\p_3),\bp(\p_4)\big)\in (\overline{{\bf H}_{\bh}^n})^4 .$$

\begin{dfn}Let $\p=(\p_1,\p_2,\p_3,\p_4)$ be a quadruple of points in $V_0\cup V_-$ such that $\bp(\p)$  are pairwise distinct points in $\overline{{\bf H}_{\bh}^n}$.   The quaternionic cross-ratio of $\p=(\p_1,\p_2,\p_3,\p_4)$ in $V_0\cup V_-$  is defined as
\begin{equation}\label{Vcross}
 \X(\p)=\X(\p_1,\p_2,\p_3,\p_4):=\langle \p_3, \p_1\rangle \langle \p_3, \p_2
 \rangle^{-1} \langle \p_4, \p_2\rangle \langle \p_4, \p_1\rangle^{-1}.
 \end{equation}
\end{dfn}
Proposition \ref{prop-1.1} implies that the definition above  is well defined.  We  mention that the quaternionic cross-ratio above has been used by Platis to obtain  the Ptolemaean inequality  in the quaternionic hyperbolic space  \cite{platis}.

Observe that, for  nonzero quaternions $\lambda_i,i=1,\cdots,4$,
\begin{equation}\label{crat1sim}
 \X(\p_1\lambda_1,\p_2\lambda_2,\p_3\lambda_3,\p_4\lambda_4)=\overline{\lambda}_1 \X(\p_1,\p_2,\p_3,\p_4)\overline{\lambda}_1^{-1}.\end{equation}
The above equality is  dominant among the properties of quaternionic cross-ratio in  $V_0\cup V_-$.

\begin{dfn} \label{crossratio} Let  $\fp=(p_1,p_2,p_3,p_4)$ be any quadruple of  pairwise distinct points in
$\overline{{\bf H}_{\bh}^n}$.  The quaternionic cross-ratio of  $\fp=(p_1,p_2,p_3,p_4)$ in $\overline{{\bf H}_{\bh}^n}$ is defined as
\begin{equation}\label{cross1}
 \X(\fp)=\X(p_1,p_2,p_3,p_4):=\X(\hat\p_1,\hat\p_2,\hat\p_3,\hat\p_4),
 \end{equation}
 where $\hat\p_i$ is the standard lift of $p_i$.
\end{dfn}

We mention that    the restriction of  Definition \ref{crossratio} on  $\partial{{\bf H}_{\bc}^n}$ is identical with the Kor\'anyi-Reimann  complex cross-ratio \cite{korrei87}. We are now ready to obtain some properties of quaternionic cross-ratios.

\begin{prop}  Let  $\fp=(p_1,p_2,p_3,p_4)$ be any quadruple of  pairwise distinct points in
$\overline{{\bf H}_{\bh}^n}$ with lifts $\p_1,\p_2,\p_3,\p_4$ and  $h\in {\rm Sp}(n,1)$.  Then
 \begin{itemize}
   \item[(i)] \quad $
\Re\big(\X(\fp)\big)=\Re\big(\X(\p)\big)=\Re\big(\X(h(p_1),h(p_2),h(p_3),h(p_4))\big);$
   \item[(ii)] \quad $|\X(\fp)|= |\X(\p)|=|\X(h(p_1),h(p_2),h(p_3),h(p_4))|.$
 \end{itemize}
 \end{prop}
\begin{proof} Noting that  $\p_i=\hat\p_i\lambda_i$ for some $\lambda_i\neq 0$ and the standard lift of $h(p_i)$ can be expressed as  $h\hat\p_i\mu_i, \mu_i\neq 0,$ we have the following two equations:
 $$\X(\p)=\X(\hat\p_1\lambda_1,\hat\p_2\lambda_2,\hat\p_3\lambda_3,\hat\p_4\lambda_4)=\overline{\lambda}_1 \X(\hat\p_1,\hat\p_2,\hat\p_3,\hat\p_4) \overline{\lambda}_1^{-1}=\overline{\lambda}_1 \X(\fp) \overline{\lambda}_1^{-1},$$
$$ \X(h(p_1),h(p_2),h(p_3),h(p_4))=\X(h\hat\p_1\mu_1,h\hat\p_2\mu_2,h\hat\p_3\mu_3,h\hat\p_4\mu_4)=\overline{\mu_1}\X(\fp)\overline{\mu}_1^{-1}.$$
The two equations above conclude the proof.
\end{proof}
We mention that those  properties  of quaternionic cross-ratio has been  used  in \cite{caopar11, kimpar03} to obtain the generalized J{\o}rgensen's inequalities in the quaternionic hyperbolic geometry.

  The following proposition displays the relationships of quaternionic cross-ratios  of a given quadruple under permutations of its points.  One should compare them with analogous results in \cite[Section 7.2]{gol99} and \cite{platis}.  The proof of the following proposition is by direct computations.
\begin{prop}
\begin{itemize}
  \item[(i)]\ \  $
\X(p_1, p_2, p_3, p_4)=\X(p_1,p_2,p_4,p_3)^{-1};$
 \item[(ii)]\ \  $
 \X(p_1, p_2, p_3,
p_4)\sim \X(p_2,p_1,p_3,p_4)^{-1};$
  \item[(iii)]\ \ $
 \X(p_1, p_2, p_3, p_4)\sim   \X(p_2,p_1,p_4,p_3) \sim
 \X(p_3,p_4,p_1,p_2)\sim  \X(p_4,p_3,p_2,p_1);$
  \item[(iv)]\begin{eqnarray*}
|\X(p_1,p_2, p_3, p_4) \X(p_1, p_4, p_2, p_3) \X(p_1, p_3, p_4, p_2)|
& = & | \X(p_1, p_2, p_3, p_4) \X(p_4, p_2, p_1, p_3) \X(p_3, p_2, p_4, p_1)|\\
& = & | \X(p_1, p_2, p_3, p_4) \X(p_4, p_2, p_3, p_1) \X(p_4, p_1, p_3, p_2)|\\
& = & | \X(p_1, p_2, p_3, p_4) \X(p_2, p_3, p_1, p_4) \X(p_3, p_1, p_2, p_4)|\\
& = & 1.
\end{eqnarray*}
\end{itemize}
\end{prop}

\medskip

\subsection{Distance formulas}

We can relate the Bergman metric with the quaternionic cross-ratio as follows. We mention that analogous results in the 2- and the 4-dimensional unit ball have been obtained in \cite[p133]{bea83} and in  \cite[Section 5]{bisgen09}.
\begin{prop}\label{prop-bdzw}
Let $z,w\in {\bf H}_{\bh}^n$ and $\gamma_{zw}$ be the   geodesic connecting $z$ and $w$
with endpoints $ u$ and $v$ in  $\partial  {\bf H}_{\bh}^n$. Then
\begin{equation*}
\rho(z,w)=|\log(\X(z, u, w, v)-1)|.
 \end{equation*}
\end{prop}

\begin{proof}  As in \cite{caopar11},  for $u,v\in \partial
{\bf H}_{\bh}^n$ with  lifts ${\bf u}$ and ${\bf v}$ such that $\langle {\bf
u}, {\bf v} \rangle=-1,$ the  {\it geodesic} $\gamma_{uv} $ in ${\bf
H}_{\bh}^n$ with endpoints $u$ and  $v$ parameterized by arc length $t$ is given by  $\bp(\gamma(t))$, where $\gamma(t)=e^{\frac{t}{2}}{\bf u}+e^{-\frac{t}{2}}{\bf v}$.  Hence we
can choose the lifts ${\bf z}, {\bf w}$ of $z,w$ as
$${\bf z}=e^{\frac{t}{2}}{\bf u}+e^{\frac{-t}{2}}{\bf v},\,{\bf w}=e^{\frac{k}{2}}{\bf u}+e^{\frac{-k}{2}}{\bf
v}.$$ In this setting, $\rho(z,w)=|t-k|$ and $\X(z, u, w, v)=\X({\bf z},{\bf u}, {\bf w}, {\bf v})=1+e^{k-t}$.
\end{proof}

As in \cite{gol99}, by Proposition \ref{prop-1.1},  for  $u, v\in \partial {\bf H}_{\bh}^n$ and
$z\in {\bf H}_{\bh}^n$,  the following quaternion \begin{equation}\eta(u, v, z) = \langle \hat{\bf u}, \hat{\bf z}\rangle \langle \hat{\bf u}, \hat{\bf v}
 \rangle^{-1} \langle \hat{\bf z}, \hat{\bf v}\rangle \langle \hat{\bf z}, \hat{\bf z}\rangle^{-1}
\end{equation}
 is well defined.
By  abuse of notation, one can view $\eta(u, v, z)$ in form as  $$ \eta(u, v, z)=\X(\hat{\bf z}, \hat{\bf v}, \hat{\bf u}, \hat{\bf z})=\X(z,v,u,z).$$
From this  we know that the real part and the modulus of $\eta(u, v, z)$ are well defined in the quaternionic setting.

By repeating almost verbatim the arguments used for the complex  case in  Propositions 7.1 and 7.6 in \cite{park},  we obtain the following result.
\begin{prop}\label{d-geo}Let  $z\in {\bf
H}_{\bh}^n, u,v\in\partial {\bf H}_{\bh}^n$ and  $\gamma_{uv}$  be the  geodesic  connecting  $u$ and $v$.
 Then the hyperbolic  distance $\rho(\gamma_{uv},z)$ from $z$ to $\gamma$ is
given by
\begin{equation}\label{d-rgeo}\cosh^2\big(\frac{\rho(\gamma_{uv},z)}{2}\big) =¯¯|\eta(u, v, z)|+ \Re\big(\eta(u, v, z)\big)\end{equation}
and  the hyperbolic distance $\rho(L_{uv},z)$ from $z$ to $L$ is given by
\begin{equation}\label{d-lgeo}\cosh^2\big(\frac{\rho(L_{uv},z)}{2}\big) =¯¯2\Re\big(\eta(u, v, z)\big).\end{equation}
\end{prop}

Let $z\in {\bf H}_{\bh}^n$,  $u,v\in\partial {\bf H}_{\bh}^n$  with  lifts  ${\bf z}, {\bf u}$ and
 ${\bf v}$ such that  $\langle {\bf u},{\bf v} \rangle=-1$. By Proposition \ref{d-geo},  the orthogonal projection $\gamma_{uv}(z)$  of $z\in {\bf H}_{\bh}^n$ to  the  geodesic $\gamma_{uv}$  endowed with hyperbolic distance $\rho$ is given by
\begin{equation}\label{prjAB}
\gamma_{uv}(z)=\bp(e^{\frac{t}{2}}{\bf u}+e^{-\frac{t}{2}}{\bf
v}),
 \end{equation}
where $e^t=|\frac{\langle {\bf v},{\bf z}\rangle}{\langle {\bf u},{\bf
z}\rangle}|$ and  ${\bf z}$  is an arbitrary lift of $z$.

\medskip

We need the following  formula of hyperbolic distance  form $z\in {\bf H}_{\bh}^n$ to the  the quaternionic line $L_{uw}$ spanned by $w\in \overline{{\bf H}_{\bh}^n}$  and $u\in   \partial {\bf H}_{\bh}^n$ later.

\begin{prop}\label{lem-dist}
Let $z\in {\bf H}_{\bh}^n$, $u\in   \partial {\bf H}_{\bh}^n$  and $v\in{\bf H}_{\bh}^n$.
Then
$$\cosh^2\big(\frac{\rho(L_{uv},z)}{2}\big)=2\Re\big(\eta(z,u,v)\big)-\frac{|\langle {\bf u},{\bf z}  \rangle|^2\langle {\bf v},{\bf v}  \rangle}{|\langle {\bf u},{\bf v}  \rangle|^2\langle {\bf z},{\bf z} \rangle}$$
and
$$\cosh^2\big(\frac{\rho(\gamma_{uv},z)}{2}\big)=\Bigr|  \X({\bf z}, {\bf v}, {\bf u}, {\bf z})-\frac{1}{2}\frac{|\langle {\bf u},{\bf z}  \rangle|^2\langle {\bf v},{\bf v}  \rangle}{|\langle {\bf u},{\bf v}  \rangle|^2\langle {\bf z},{\bf z} \rangle} \Bigl|+\Re\big(\eta(z,u,v)\big)-\frac{1}{2}\frac{|\langle {\bf u},{\bf z}  \rangle|^2\langle {\bf v},{\bf v}  \rangle}{|\langle {\bf u},{\bf v}  \rangle|^2\langle {\bf z},{\bf z} \rangle},$$
where ${\bf z}, {\bf u}$ and ${\bf v}$ and   are lifts of $z,u,v$, respectively.
\end{prop}

 \begin{proof}  Let $w$ be the other endpoint of the geodesic $\gamma_{uv}$ and  ${\bf u}$ and ${\bf w}$  be lifts of $u$ and $w$ such that $\langle {\bf u},{\bf w}  \rangle=-1$. Then there exist a $t\in \br$ and a lift  $\bf v$ of $v$ such that ${\bf v}=e^{\frac{t}{2}}{\bf u}+e^{-\frac{t}{2}}{\bf w}.$  Therefore
 \begin{equation}\label{extprop}{\bf w}=e^{\frac{t}{2}}{\bf v}-e^{t}{\bf u}, \ \langle {\bf u},{\bf v}  \rangle=-e^{-\frac{t}{2}}, \langle {\bf v},{\bf v}  \rangle=-2.\end{equation} Let $\mu= \X({\bf z}, {\bf w}, {\bf u}, {\bf z}).$  Then  by (\ref{extprop}) we can express $\mu$ as
\begin{eqnarray*}\mu&=&\langle {\bf u},{\bf z}\rangle \langle {\bf u},{\bf v}\rangle^{-1} \langle {\bf z},{\bf v}\rangle \langle {\bf z},{\bf z}\rangle^{-1}-\frac{1}{2}\frac{|\langle {\bf u},{\bf z}  \rangle|^2\langle {\bf v},{\bf v}  \rangle}{|\langle {\bf u},{\bf v}  \rangle|^2\langle {\bf z},{\bf z} \rangle}\\
&=& \X({\bf z}, {\bf v}, {\bf u}, {\bf z})-\frac{1}{2}\frac{|\langle {\bf u},{\bf z}  \rangle|^2\langle {\bf v},{\bf v}  \rangle}{|\langle {\bf u},{\bf v}  \rangle|^2\langle {\bf z},{\bf z} \rangle}.
\end{eqnarray*}
It is obvious that  $$\cosh^2\big(\frac{\rho(L_{uv},z)}{2}\big)= \cosh^2\big(\frac{\rho(L_{uw},z)}{2}\big),\ \ \cosh^2\big(\frac{\rho(\gamma_{uv},z)}{2}\big)=\cosh^2\big(\frac{\rho(\gamma_{uw},z)}{2}\big).$$  Note that $|\mu|$ and $\Re(\mu)$  are independent of the choice of lifts ${\bf z}, {\bf u}, {\bf v}$. The result follows from Proposition \ref{d-geo}. \end{proof}

\medskip

\subsection{Quaternionic Cartan's angular  invariant}

Recall from Section 1 that given  a triple  $\fp=(p_1,p_2,p_3)$   of  pairwise distinct points in $\overline{{\bf H}_{\bh}^n}$ with lifts  $\p_1, \p_2, \p_3$, then
\begin{equation} \label{angulars} \ba_{\bh}(\fp)= \ba_{\bh}(p_1,p_2,p_3)=\arccos \frac{\Re(-\langle \p_1, \p_2, \p_3\rangle)}{|\langle \p_1, \p_2, \p_3\rangle|}\end{equation}
is  the quaternionic Cartan's angular  invariant associated to $\fp$.
  This is the  angle  between the real axis and the radius vector of $\langle \p_1, \p_2,\p_3\rangle$  used by  Apanasov and Kim \cite{apakim07}.

As in the proof of Proposition \ref{prop-1.2}, we have
\begin{equation}\label{angularslif}  \ba_{\bh}(\fp)=\arccos \frac{\Re(-\langle \p_1\lambda_1,\p_2\lambda_2,\p_3\lambda_3\rangle)}{|\langle \p_1\lambda_1,\p_2\lambda_2,\p_3\lambda_3\rangle|}\end{equation}
and
$$ \ba_{\bh}(p_1,p_2,p_3)=\ba_{\bh}(p_{\iota(1)},p_{\iota(2)},p_{\iota(3)}),$$
 where $\iota$ is a permutation of $1,2,3$ and $\lambda_i\in \bh\setminus \{0\}$.
 Therefore it makes sense to define \begin{equation} \label{angularsv} \ba_{\bh}(\p)= \ba_{\bh}(\p_1,\p_2,\p_3)=\arccos \frac{\Re(-\langle \p_1, \p_2, \p_3\rangle)}{|\langle \p_1, \p_2, \p_3\rangle|}.\end{equation}
 We can verify that  $$\ba_{\bh}(\p)=\ba_{\bh}(\p_1,\p_2,\p_3)=\ba_{\bh}(\p_1\lambda_1,\p_2\lambda_2,\p_3\lambda_3)
 =\ba_{\bh}(\p_{\iota(1)}\lambda_{\iota(1)},\p_{\iota(2)}\lambda_{\iota(2)},\p_{\iota(3)}\lambda_{\iota(3)})$$
and \begin{equation} \ba_{\bh}(\fp)=\ba_{\bh}(g(p_1),g(p_2),g(p_3)),\forall g\in  {\rm Sp}(n,1).\end{equation}

The following properties are in \cite{apakim07} .
\begin{prop}\label{proptrip} Let $\fp=(p_1,p_2,p_3)$ be a triple of pairwise distinct points in   $\partial{\bf H}_{\bh}^n$. Then
\begin{itemize}
 \item[(i)] Three points $p_1,p_2,p_3$ lie in the same $\br$-circle if and only if  $\ba_{\bh}(\fp)=0$.
 \item[(ii)] Three points $p_1,p_2,p_3$ lie in the boundary of an $\bh$-line  if and only if  $\ba_{\bh}(\fp)=\pi/2$.
  \end{itemize}
  \end{prop}

The quaternionic Cartan's angular invariant  enjoys the following geometric interpretation.  This geometric interpretation is a slight generalisation of Theorem 3.4 in \cite{apakim07}.

\begin{prop}\label{prop-angdist}
For distinct points $z, w\in \partial {\bf H}_{\bh}^n$ and $r\in \overline{{\bf H}_{\bh}^n}$,  let $\Pi:
{\bf H}_{\bh}^n\to L_{zw}$  and $\Theta:{\bf H}_{\bh}^n\to \gamma_{zw}$  be the orthogonal projections of the space
${\bf H}_{\bh}^n$ endowed with hyperbolic distance $\rho$. Then
\begin{itemize}
  \item[(i)] $$\Theta \Pi r=\Theta r, \ \  \tan \ba_{\bh}(z,w,r) = \sinh(\rho(\gamma_{zw},\Pi r)).$$
  \item[(ii)] If $r\in {\bf H}_{\bh}^n$,  then we  have  $$\cosh\big(\frac{\rho(\gamma_{zw},r)}{2}\big)=\cosh\big(\frac{\rho(L_{zw},r)}{2}\big)\cosh\big(\frac{\rho(\gamma_{zw},\Pi r)}{2}\big)$$  and $$\tan \ba_{\bh}(z,w,r)=\frac{2\sqrt{2} \cosh\big(\frac{\rho(\gamma_{zw},r)}{2}\big)\sqrt{\cosh\rho(\gamma_{zw},r)-\cosh\rho(L_{zw},r)}}{1+\cosh\rho(L_{zw},r) }.$$
\end{itemize}
\end{prop}
\begin{proof}   Since ${\rm Sp}(n,1)$ acts doubly transitively on $\partial {\bf H}_{\bh}^n$,  we may assume that $z=o, w=\infty, r=(r_1,\cdots,r_n)^T\in\overline{{\bf H}_{\bh}^n}$.  Then $$L_{zw}=\bp(\{\hat{\bf o}\lambda+ \hat{\bf \infty}\mu,  \lambda,\mu,\in \bh\})\cap \overline{{\bf H}_{\bh}^n}=\{(x,0,\cdots,0)^T: \Re(x)\leq 0 \}.$$
  Therefore $\Pi r=(r_1,0,\cdots,0)^T$.   It follows from (\ref{prjAB}) that
  $$\Theta \Pi r=\Theta r=(-|r_1|,0,\cdots,0).$$
  Since $\langle \hat{\bf o}, \hat{\bf \infty}, {\hat{\bf r}}\rangle=\overline{r_1}$,  we have  $\cos\ba_{\bh}(\fp)=\frac{-\Re(r_1)}{|r_1|}$ and therefore  \begin{equation}\label{add2}\tan \ba_{\bh}(z,w,r)=\frac{|\Im(r_1)|}{-\Re(r_1)}.\end{equation}   Noting that $\eta(o,\infty,\Pi r)=\frac{\overline{r_1}}{2\Re(r_1)}$, by Proposition \ref{d-geo} we get \begin{equation}\label{propeqg1}\cosh^2\big(\frac{\rho(\gamma_{zw},\Pi r)}{2}\big)=\frac{\Re(r_1)-|r_1|}{2\Re(r_1)}.\end{equation}
 Therefore
 \begin{equation}\label{add1}\sinh^2 (\rho(\gamma_{zw},\Pi r))=[2\cosh^2\big(\frac{\rho(\gamma_{zw},\Pi r)}{2}\big)-1]^2-1=\frac{|\Im(r_1)|^2}{\Re(r_1)^2}.\end{equation}
This concludes the proof of (i).

If $r\in {\bf H}_{\bh}^n$ then $2\Re(r_1)+\sum_{i=2}^{n}|r_i|^2<0$ and $$\eta(o,\infty,r)=\frac{\bar{r_1}}{2\Re(r_1)+\sum_{i=2}^{n}|r_i|^2}.$$
It follows from  Proposition \ref{d-geo}  that  \begin{equation}\cosh^2\big(\frac{\rho(\gamma_{zw},r)}{2}\big)=\frac{\Re(r_1)-|r_1|}{2\Re(r_1)+\sum_{i=2}^{n}|r_i|^2},\ \ \cosh^2\big(\frac{\rho(L_{zw},r)}{2}\big)=\frac{2\Re(r_1)}{2\Re(r_1)+\sum_{i=2}^{n}|r_i|^2}. \end{equation}
 Hence
\begin{equation}\label{propeqg2}\cosh\big(\frac{\rho(\gamma_{zw},r)}{2}\big)=\cosh\big(\frac{\rho(L_{zw},r)}{2}\big)\cosh\big(\frac{\rho(\gamma_{zw},\Pi r)}{2}\big). \end{equation}
 By (\ref{propeqg2}) we get
$$\sinh^2 (\rho(\gamma_{zw},\Pi r))=\frac{8\cosh^2\big(\frac{\rho(\gamma_{zw},r)}{2}\big)\big(\cosh\rho(\gamma_{zw},r)-\cosh\rho(L_{zw},r)\big)}{\big(1+\cosh\rho(L_{zw},r)\big)^2 }.$$
The above equation,  together with (\ref{add2}) and (\ref{add1}),  concludes the proof of (ii).
\end{proof}

The following proposition relates the cyclic product of quaternionic cross-ratios to the quaternionic Cartan's angular  invariant.  One can compare it with  analogous result of complex case in  \cite[p225]{gol99}.
\begin{prop}
 \begin{equation*}
\X(p_1, p_2, p_3, p_4)\X(p_1, p_4, p_2, p_3)\X(p_1, p_3, p_4,
p_2)=\exp(2\ba_{\bh}(p_2,p_3,p_4) {\bf J}),
 \end{equation*}
 where ${\bf J}$ is the imaginary part of $\X(p_1, p_2, p_3, p_4)\X(p_1, p_4, p_2, p_3)\X(p_1, p_3, p_4,
p_2)$. \end{prop}

\begin{proof}  Let $Y=\X(p_1, p_2, p_3, p_4)\X(p_1, p_4, p_2, p_3)\X(p_1, p_3, p_4,
p_2)$.  Then $|Y|=1$ and
\begin{eqnarray*}Y&=&\langle{\hat\p_3},\,{\hat\p_1}\rangle
\langle{\hat\p_3},\,{\hat\p_2}\rangle^{-1} \langle{\hat\p_4},\,{\hat\p_2}\rangle \langle{\hat\p_4},\,{\hat\p_1}\rangle^{-1}  \langle{\hat\p_2},\,{\hat\p_1}\rangle \langle{\hat\p_2},\,{\hat\p_4}\rangle^{-1}
\langle{\hat\p_3},\,{\hat\p_4}\rangle \langle{\hat\p_3},\,{\hat\p_1}\rangle^{-1} \\
& & \langle{\hat\p_4},\,{\hat\p_1}\rangle \langle{\hat\p_4},\,{\hat\p_3}\rangle^{-1} \langle{\hat\p_2},\,{\hat\p_3}\rangle \langle{\hat\p_2},\,{\hat\p_1}\rangle^{-1}.\end{eqnarray*}
 Hence  \begin{equation*}
\Re(Y)=\Re\Big(\frac{\langle {\hat\p_2}, {\hat\p_3}, {\hat\p_4}\rangle^2}{|\langle {\hat\p_2}, {\hat\p_3}, {\hat\p_4}\rangle|^2}\Big).
 \end{equation*}
 Noting that
\begin{equation*} \frac{-\langle {\hat\p_2}, {\hat\p_3}, {\hat\p_4}\rangle}{|\langle {\hat\p_2}, {\hat\p_3}, {\hat\p_4}\rangle|} =
\exp(\ba_{\bh}(p_2,p_3,p_4) {\bf I}),\,\,  {\bf I}=\frac{-\Im(\langle
{\hat\p_2}, {\hat\p_3}, {\hat\p_4}\rangle)}{|\Im(\langle {\hat\p_2},
{\hat\p_3}, {\hat\p_4}\rangle)|},
\end{equation*}
we have
\begin{equation*} \frac{\langle {\hat\p_2}, {\hat\p_3}, {\hat\p_4}\rangle^2}{|\langle {\hat\p_2}, {\hat\p_3}, {\hat\p_4}\rangle|^2} =
\exp(2\ba_{\bh}(p_2,p_3,p_4) {\bf I}).
\end{equation*}
\end{proof}

\section{The  congruence class of  triple  of pairwise distinct points  of  $\overline{{\bf H}_{\bh}^n}$}\label{cong-trip}
 Denote by $$ G_{o,\infty}=\{g\in {\rm Sp}(n,1):g(o)=o, \  g(\infty)=\infty\}.$$  It follows from  (\ref{hform}) that  each $h\in G_{o,\infty}$ is of the form $$h={\rm diag}(\mu, A, \bar{\mu}^{-1}),$$ where $A\in {\rm Sp}(n-1)$.   Moveover, if $|\mu|=1$ then $h$ is a boundary elliptic element fixing  the geodesic $\gamma_{o\infty}$ pointwise.

We  need the following  lemma to prove  Theorem \ref{thm-triple}.

\begin{lem}\label{lem-map}
Let $z=(z_1,\cdots,z_n)^T$ and $w=(w_1,\cdots,w_n)^T$ be points of $\overline{{\bf H}_{\bh}^n}.$ Then

\begin{itemize}
  \item[(i)]  there exists  an $h\in G_{o,\infty}$ such that $w=h(z)$ if and only if this exists a  $\kappa>0$ such that \begin{equation}\label{condm}\Re(w_1)=\kappa^2\Re(z_1),\ \  |w_1|=\kappa^2|z_1|, \ \ \sum_{i=2}^{n}|w_i|^2=\kappa^2\sum_{i=2}^{n}|z_i|^2;\end{equation}
  \item[(ii)] there exists an elliptic element $h$ fixing  the geodesic $\gamma_{o\infty}$ pointwise  such that $w=h(z)$ if and only if \begin{equation}\label{condme}\Re(w_1)=\Re(z_1),\ \  |w_1|=|z_1|, \ \ \sum_{i=2}^{n}|w_i|^2=\sum_{i=2}^{n}|z_i|^2.\end{equation}
  \end{itemize}
\end{lem}

\begin{proof}  We  first consider  Case (i).  Let $h={\rm diag}(\mu, A, \bar{\mu}^{-1})$, where $A\in {\rm Sp}(n-1)$.  It follows from $w=h(z)$ that
 $h{\hat{\bf z}}={\hat{\bf w}}\bar{\mu}^{-1}$, that is,
 $$ \mu z_1=w_1 \bar{\mu}^{-1}, A(z_2,\cdots,z_n)^T=(w_2,\cdots,w_n)^T \bar{\mu}^{-1}.$$
 Therefore the condition (\ref{condm}) holds.

 For  sufficiency,  if the condition (\ref{condm}) holds then there exist a unit quaternion $\lambda$ and  an  $A\in {\rm Sp}(n-1)$  such that
\begin{equation*}
\lambda(\kappa^2z_1)\lambda^{-1}=w_1,\,\, A(z_2,\cdots,z_n)^T=(w_2,\cdots,w_n)^T\frac{\lambda}{\kappa}.
\end{equation*}
Therefore $h={\rm diag}(\kappa \lambda, A, \frac{\lambda}{\kappa})\in  G_{o,\infty}$ is the desired isometry mapping $z$ to $w$.

Similarly we can prove  Case (ii).
\end{proof}

\medskip

We now consider the congruence class of  triple  of pairwise distinct points  of  $\overline{{\bf H}_{\bh}^n}$.

\medskip

\noindent {\it Proof  of Theorem \ref{thm-triple}.}\quad   The  necessity is obvious.    We mention that the geometric  invariants in Theorem \ref{thm-triple} are invariant under the conjugation by  elements of ${\rm Sp}(n,1)$.  For  sufficiency,  we can choose $f,g\in {\rm Sp}(n,1)$  to map the geodesics connecting $p_1$ and $p_2$, $q_1$ and $q_2$  to $\gamma_{o\infty}$, respectively. By our condition, we can further assume that $f(p_1)=g(q_1)$ and $f(p_2)=g(q_2)$.  Then we need to find an $h\in {\rm Sp}(n,1)$ which fixes $f(p_1)$ and $f(p_2)$ and  maps $f(p_3)$ to $g(q_3)$.

 \medskip

 Let $p_3=(r_1,\cdots,r_n)^T,  q_3=(z_1,\cdots,z_n)^T\in \overline{{\bf H}_{\bh}^n}.$   By the above normalisation, we only need to consider the following  four  specific cases.

 \medskip

 Case (i) and  Case (ii):  $\fp=(o,\infty,p_3)$ and $\fq=(o,\infty, q_3)$, $p_3, q_3\in \overline{{\bf H}_{\bh}^n}.$

 Since $\langle \hat{\bf o}, \hat{\bf \infty}, {\hat\p_3}\rangle=\overline{r_1}$ and $\ba_{\bh}(\fp)=\ba_{\bh}(\fq)$, we get \begin{equation}\label{c1A}\frac{\Re(r_1)}{|r_1|}=\frac{\Re(z_1)}{|z_1|}.\end{equation}
For Case (i), that is  $p_3,q_3\in \partial {\bf H}_{\bh}^n$, we have   $-2\Re(r_1)=\sum_{i=2}^{n}|r_i|^2$ and $-2\Re(z_1)=\sum_{i=2}^{n}|z_i|^2$.
 Therefore the condition (\ref{condm}) holds.

For Case (ii), that is  $p_3,q_3\in {\bf H}_{\bh}^n$, we have  $-2\Re(r_1)>\sum_{i=2}^{n}|r_i|^2$ and $-2\Re(z_1)>\sum_{i=2}^{n}|z_i|^2$.  Since $$\eta(o,\infty,p_3)=\frac{\bar{r_1}}{2\Re(r_1)+\sum_{i=2}^{n}|r_i|^2}$$
and $\rho(L_{o\infty},p_3) =\rho(L_{o\infty},q_3)$, by Proposition \ref{d-geo} we have
\begin{equation}\label{c1p}\frac{2\Re(r_1)}{2\Re(r_1)+\sum_{i=2}^{n}|r_i|^2}=\frac{2\Re(z_1)}{2\Re(z_1)+\sum_{i=2}^{n}|z_i|^2}. \end{equation}
 Thus the condition (\ref{condm}) holds.  Lemma \ref{lem-map}  concludes the proof of Cases (i) and (ii).

\vspace{2mm}

\noindent {\bf Remark 4.1.} \   In Case (ii), by Proposition \ref{prop-angdist} we can use  $\rho(\gamma_{p_1p_2},p_3)=\rho(\gamma_{q_1q_2},q_3)$  to replace the condition $\rho(L_{p_1p_2},p_3)=\rho(L_{q_1q_2},q_3).$   Let  $\fp=(o,\infty,p_3)$ and $\fq=(o,\infty, q_3)$, where $p_3=(-1,\frac{\sqrt{6}}{2})\in  {\bf H}_{\bc}^2$ and $q_3=(-1,\sqrt{2}{\bf i})\in \partial {\bf H}_{\bc}^2$. Though $\ba_{\bh}(\fp)=\ba_{\bh}(\fq)=0$, we can not map $p_3$ to $q_3$ by any  $f\in  {\rm Sp}(n,1)$. This implies the condition $\rho(L_{p_1p_2},p_3)=\rho(L_{q_1q_2},q_3)$ is necessary.

 \medskip

  Case (iii):  $p_1\in \partial {\bf H}_{\bh}^n$, $p_2, p_3\in {\bf H}_{\bh}^n$.

Since ${\rm Sp}(n,1)$  acts transitively on  $\partial {\bf H}_{\bh}^n\times {\bf H}_{\bh}^n$, we may assume that $p_1=q_1=o$ and  $p_2=q_2=(-1,0,\cdots,0)^T=\bp( \hat{\bf \infty}+ \hat{\bf o})$. In this case the other end point of the  geodesic connecting $o$ and $p_2$ is $\infty$.
It follows from $\rho(L_{o\infty},p_3)=\rho(L_{o\infty},q_3)$ that
\begin{equation}\label{c3L}\frac{\Re(r_1)}{2\Re(r_1)+\sum_{i=2}^{n}|r_i|^2}=\frac{\Re(z_1)}{2\Re(z_1)+\sum_{i=2}^{n}|z_i|^2}. \end{equation}
By $\rho(p_2,p_3)=\rho(q_2,q_3)$ we get
\begin{equation}\label{c3d}\frac{|r_1-1|^2}{2\Re(r_1)+\sum_{i=2}^{n}|r_i|^2}=\frac{|z_1-1|^2}{2\Re(z_1)+\sum_{i=2}^{n}|z_i|^2}.  \end{equation}
It follows  from  (\ref{c3L}) and (\ref{c3d}) that
\begin{equation}\label{cc1} \frac{\sum_{i=2}^{n}|r_i|^2}{\Re(r_1)}=\frac{\sum_{i=2}^{n}|z_i|^2}{\Re(z_1)}, \frac{\Re(r_1)}{|1-r_1|^2}=\frac{\Re(z_1)}{|1-z_1|^2}, \ \  \frac{1+|r_1|^2}{\Re(r_1)}=\frac{1+|z_1|^2}{\Re(z_1)}.  \end{equation}
 By Proposition \ref{lem-dist} we have
$$\cosh^2\big(\frac{\rho(L_{or},p_2)}{2}\big)=1+\frac{\sum_{i=2}^{n}|r_i|^2}{2|r_1|^2}, \  \cosh^2\big(\frac{\rho(L_{oq_3},q_2)}{2}\big)=1+\frac{\sum_{i=2}^{n}|z_i|^2}{2|z_1|^2}.$$
Hence
\begin{equation}\label{cc2} \frac{\sum_{i=2}^{n}|r_i|^2}{|r_1|^2}=\frac{\sum_{i=2}^{n}|z_i|^2}{|z_1|^2}.\end{equation}
By (\ref{cc1}) and (\ref{cc2}) we get $\Re(r_1)=\Re(z_1), |z_1|=|r_1|, \ \ \sum_{i=2}^{n}|z_i|^2=\sum_{i=2}^{n}|r_i|^2.$
Lemma \ref{lem-map}  concludes the proof of Case (iii).

\vspace{2mm}

\noindent {\bf Remark 4.2.}  \  Since $\langle \hat{\bf o}, \hat{\bf o}+ \hat{\bf \infty}, {\hat\p_3}\rangle=\overline{r_1}-|r_1|^2$, the condition $\ba_{\bh}(\fp)=\ba_{\bh}(\fq)$ implies that \begin{equation}\label{c3A}\frac{\Re(r_1)-|r_1|^2}{|r_1||1-r_1|}=\frac{\Re(z_1)-|z_1|^2}{|z_1||1-z_1|}.\end{equation}
Let $p_3=(-1+{\bf i},\sqrt{6}/2),q_3=(-2/5+{\bf i}/5, \sqrt{15}/5)$.  Then the conditions (\ref{c3L}),(\ref{c3d}) and (\ref{c3A}) hold for $\fp=(o,\bp( \hat{\bf \infty}+ \hat{\bf o}),p_3)$ and $\fq=(o,\bp( \hat{\bf \infty}+ \hat{\bf o}), q_3)$. If there exist an  $h\in {\rm Sp}(2,1)$  mapping $\fp$ to  $\fq$ then $h$ fixes $o$ and $\bp( \hat{\bf \infty}+ \hat{\bf o})$.  This implies that $h$ is a boundary elliptic element fixing the geodesic $\gamma_{o\infty}$ pointwise. However Lemma \ref{lem-map}(ii) indicates that there exist not such an  $h\in {\rm Sp}(2,1)$ for $-1\neq -2/5$. This example shows that the quaternionic Cartan's angular  invariant is not a good candidate for congruence class in this case.

  \medskip

   Case (iv): $p_1,p_2, p_3\in {\bf H}_{\bh}^n$

   Mapping the  geodesic $\gamma_{p_1p_2}$ and $\gamma_{q_1q_2}$  by $h_1, h_2 \in  {\rm Sp}(n,1)$ to the geodesic $\gamma_{o \infty}$, we may assume that
   $$p_1=q_1=\bp( \hat{\bf \infty}+ \hat{\bf o}),\ \ p_2=q_2=\bp( e^{\frac{t}{2}}\hat{\bf \infty}+ e^{-\frac{t}{2}}\hat{\bf o})=(-e^t,0,\cdots,0),$$
   where $t=\rho(p_1,p_2)$.

Since $$\langle \hat{\bf \infty}+ \hat{\bf o},e^{\frac{t}{2}}\hat{\bf \infty}+ e^{-\frac{t}{2}}\hat{\bf o}, {\hat\p_3}\rangle=(e^{\frac{t}{2}}+e^{-\frac{t}{2}})(-e^{\frac{t}{2}}+e^{-\frac{t}{2}}r_1)(1-\overline{r_1}),$$  the condition  $\ba_{\bh}(\fp)=\ba_{\bh}(\fq)$  implies  \begin{equation}\label{c4A}\frac{\Re[(r_1-e^t)(1-\bar{r_1})]}{|(r_1-e^t)(1-\bar{r_1})|}=\frac{\Re[(z_1-e^t)(1-\bar{z_1})]}{|(z_1-e^t)(1-\bar{z_1})|},\end{equation}
that is
\begin{equation}\label{c4A1}\frac{(e^t+1)\Re(r_1)-|r_1|^2-e^t}{|(r_1-e^t)(1-\bar{r_1})|}=\frac{(e^t+1)\Re(z_1)-|z_1|^2-e^t}{|(z_1-e^t)(1-\bar{z_1})|}.\end{equation}
It follows from  $\rho(p_1,p_3)=\rho(q_1,q_3)$ that
\begin{equation}\label{c4d}\frac{|r_1-1|^2}{2\Re(r_1)+\sum_{i=2}^{n}|r_i|^2}=\frac{|z_1-1|^2}{2\Re(z_1)+\sum_{i=2}^{n}|z_i|^2}.  \end{equation}
 Similarly,  $\rho(p_2,p_3)=\rho(q_2,q_3)$ implies that
\begin{equation}\label{c4d2}\frac{|r_1-e^t|^2}{2\Re(r_1)+\sum_{i=2}^{n}|r_i|^2}=\frac{|z_1-e^t|^2}{2\Re(z_1)+\sum_{i=2}^{n}|z_i|^2}.  \end{equation}
From equations (\ref{c4A1})-(\ref{c4d2}) we obtain that
\begin{equation}\label{c4A2}\frac{(e^t+1)\Re(r_1)-|r_1|^2-e^t}{|1-r_1|^2}=\frac{(e^t+1)\Re(z_1)-|z_1|^2-e^t}{|1-z_1|^2}.\end{equation}
It follows from  the above formula that
\begin{equation}\label{c4A3}\frac{2\Re(r_1)-1}{|r_1|^2}=\frac{2\Re(z_1)-1}{|z_1|^2}.\end{equation}
By (\ref{c4d}) and (\ref{c4d2}) we obtain that $$\frac{|r_1-1|^2}{|r_1-e^t|^2}=\frac{|z_1-1|^2}{|z_1-e^t|^2}.$$  We  can rephrase the above equation as
\begin{equation}\label{c4A4}2|z_1|^2\Re(r_1)-(e^t+1)|z_1|^2+2e^t\Re(z_1)=2|r_1|^2\Re(z_1)-(e^t+1)|r_1|^2+2e^t\Re(r_1).\end{equation}
Substituting (\ref{c4A3}) into (\ref{c4A4}) we obtain
\begin{equation}\label{c4A5}|r_1|^2-2\Re(r_1)=|z_1|^2-2\Re(z_1).\end{equation}
Thus \begin{equation}\label{c4A6}|r_1|^2\Bigr(1-\frac{2\Re(r_1)-1}{|r_1|^2}\Bigl)=|z_1|^2\Bigr(1-\frac{2\Re(z_1)-1}{|z_1|^2}\Bigl).\end{equation}
This implies $|z_1|=|r_1|$ and therefore  $\Re(r_1)=\Re(z_1),   \ \sum_{i=2}^{n}|z_i|^2=\sum_{i=2}^{n}|r_i|^2.$  Lemma \ref{lem-map}  concludes the proof of
Case (iv).

\vspace{2mm}

\noindent {\bf Remark 4.3.} \  In Case (iv), since $p_1$ and $p_2$ lie in  the geodesic $\gamma_{o\infty}$,  we have $L_{o\infty}=L_{p_1p_2}$.  It follows from $\rho(L_{o\infty},r)=\rho(L_{o\infty},q_3)$ that
\begin{equation}\label{c4p}\frac{\Re(r_1)}{2\Re(r_1)+\sum_{i=2}^{n}|r_i|^2}=\frac{\Re(z_1)}{2\Re(z_1)+\sum_{i=2}^{n}|z_i|^2}.  \end{equation}
The equations (\ref{c4d}),  (\ref{c4d2})  and  (\ref{c4p}) imply that the condition (\ref{condme}) holds. This observation implies that we can replace the condition  $\ba_{\bh}(\fp)=\ba_{\bh}(\fq)$ by $\rho(L_{p_1p_3},p_2)=\rho(L_{q_1q_3},q_2)$ in Case (iv) of Theorem \ref{thm-triple}.
\hfill$\square$

\section{The moduli space of quadruples of pairwise distinct points  of  $\partial {\bf H}_{\bh}^n$}\label{sect-moduli}

\subsection{The Gram matrix}

\quad \ Given a quadruple $\mathfrak{p}=(p_1,p_2,p_3,p_4)$  in $\partial {\bf H}_{\bh}^n$ with lift  ${\bf p}=({\bf p}_{1},{\bf p}_{2},{\bf p}_{3}, {\bf p}_{4})$.
The following  Hermitian  matrix
$$G=G({\bf p})=(g_{ij})=( \langle {\bf p}_{i},{\bf p}_{j}\rangle)$$
is called the  {\it Gram matrix}  associated to $\mathfrak{p}$.
It is obvious that \begin{equation}\label{isomap}G({\bf p})=G(f{\bf p})=G(f{\bf p}_{1},f{\bf p}_{2},f{\bf p}_{3}, f{\bf p}_{4}),\  f\in {\rm PSp}(n,1).\end{equation}
Let  $D={\rm diag}(\lambda_1,\lambda_2,\lambda_3,\lambda_4)$ and $\tilde{{\bf p}}= ({\bf p}_{1}\lambda_1,{\bf p}_{2}\lambda_2,{\bf p}_{3}\lambda_3, {\bf p}_{4}\lambda_4)$.  Then  $$\tilde{G}=G(\tilde{{\bf p}})=( \langle {\bf p}_{i}\lambda_i,{\bf p}_{j}\lambda_j\rangle)=(\bar{\lambda_j}\langle {\bf p}_{i},{\bf p}_{j}\rangle\lambda_i)$$ and
$$\overline{\tilde{G}}=D^*\overline{G}D,$$
where $\overline{G}=(\bar{g_{ij}})$ is the conjugate of $G$.

We say that two Hermitian matrices $H$ and $\tilde{H}$ are equivalent if there
exists a diagonal matrix
 $$D={\rm diag}(\lambda_1,\lambda_2,\lambda_3, \lambda_4),\  \lambda_i\in \bh\setminus \{0\}$$ such that $$\tilde{H} = D^*HD.$$
 Thus, to each quadruple  ${\bf p}$ of points in $V_0$ is associated an equivalence class
of Hermitian matrices with zeros on the diagonal.

\begin{prop}\label{normprocess}
Let  $\mathfrak{p}=(p_1,p_2,p_3,p_4)$ be a quadruple of pairwise distinct points in $\partial {\bf H}_{\bh}^n$.  Then the equivalence class of Gram matrices associated to $\mathfrak{p}$ contains a unique  matrix $G= (g_{ij})$ with $$g_{ii} = 0,i=1,\cdots 4,\   g_{12}=g_{23}=g_{34}=1,$$ and
  $$g_{13}=-e^{{\bf i}\ba},\ g_{14}=c_1+t{\bf j},\ g_{24}=c_2+c_3{\bf j},$$  where $\ba=\ba_{\bh}((p_1,p_2,p_3))$, $t\geq 0$  and $c_i\in \bc, i=1,2,3$.
  Furthermore $c_3\geq 0$ provided $t=0$.
\end{prop}

\begin{proof}  Let ${\bf p}=({\bf p}_{1},{\bf p}_{2},{\bf p}_{3}, {\bf p}_{4})$ be an arbitrary lift of $\mathfrak{p}$.   We want to obtain a lift $\n$ of $\mathfrak{p}$ by rescaling ${\bf p}$ such that $G(\bf n)$  is the desired Gram matrix.

Note that $\langle {\bf p}_{i}, {\bf p}_{j}\rangle\neq 0$ for $i\neq j$.   Firstly we  obtain  $\lambda_i, i=2,3,4$  as the solutions of the following equations:
\begin{equation}\label{findlambda} \langle {\bf p}_{1}, \ {\bf p}_{2}\lambda_2\rangle=1,\ \langle {\bf p}_{2} \lambda_2, \ {\bf p}_{3}\lambda_3\rangle=1, \  \langle {\bf p}_{3}\lambda_3, {\bf p}_{4}\lambda_4\rangle =1. \end{equation} These solutions  $\lambda_i, i=2,3,4$ are given by (\ref{flam2-4}) in terms of $\p_i, i=1,2,3,4$.
Now the Gram matrix  $G$ of $({\bf p}_{1},\  {\bf p}_{2}\lambda_2,\  {\bf p}_{3}\lambda_3,\  {\bf p}_{4}\lambda_4)$ satisfies $g_{ii} = 0,i=1,\cdots 4,\   g_{12}=g_{23}=g_{34}=1$.

Secondly, we want to find $\lambda_1$ such that the Gram matrix  $G=(g_{ij})$ of $$({\bf p}_{1}\lambda_1,\  {\bf p}_{2}\lambda_2\bar{\lambda_1}^{-1},\  {\bf p}_{3}\lambda_3 \lambda_1,\  {\bf p}_{4}\lambda_4 \bar{\lambda_1}^{-1})$$ satisfies $g_{ii} = 0,\ i=1,\cdots 4,\   g_{12}=g_{23}=g_{34}=1,\  g_{13}=-e^{{\bf i}\ba}$.
We mention that the requirement $$g_{13}=\langle {\bf p}_{1} \lambda_1, \ {\bf p}_{3}\lambda_3\lambda_1\rangle=-e^{{\bf i}\ba}$$ comes from Proposition \ref{prop-1.2} and  the property (\ref{angularslif}).  We can verify that  $\lambda_1$   of the form  (\ref{flam1}) is the desired solution.

 We mention that $\lambda_1$ given by (\ref{flam1})  is just a specific choice.   For example, let $\lambda_1'=\lambda_1 e^{{\bf i}\theta}$ for arbitrary $\theta$  and $\lambda_i'=\lambda_i, \ i=2,3,4$. Then  the Gram matrix  $G'=(g_{ij}')$ of $({\bf p}_{1}\lambda_1',\  {\bf p}_{2}\lambda_2'\bar{\lambda_1'}^{-1},\  {\bf p}_{3}\lambda_3' \lambda_1',\  {\bf p}_{4}\lambda_4' \bar{\lambda_1'}^{-1})$ also  satisfies $g_{ii}' = 0,\ i=1,\cdots 4,\   g_{12}'=g_{23}'=g_{34}'=1,\  g_{13}'=-e^{{\bf i}\ba}$.   However  $$g_{14}'= e^{-{\bf i}\theta}g_{14} e^{{\bf i}\theta},\ \   g_{24}' =e^{-{\bf i}\theta}g_{24} e^{{\bf i}\theta}.$$
 It is  the function  $\sigma(a,b)$ which eliminates this indeterminacy in the following third step.

 Thirdly we set \begin{equation}\label{fmuprop} \mu= \sigma(\langle {\bf p}_{1}\lambda_1, {\bf p}_{4}\lambda_4 \bar{\lambda_1}^{-1} \rangle,  \langle {\bf p}_{2}\lambda_2 \bar{\lambda_1}^{-1}, {\bf p}_{4}\lambda_4 \bar{\lambda_1}^{-1}\rangle).\end{equation}
Then the Gram matrix of lift  \begin{equation}\label{findnormlift}{\bf n}=({\bf n}_{1},{\bf n}_{2},{\bf n}_{3}, {\bf n}_{4})=({\bf p}_{1}\lambda_1\mu,\  {\bf p}_{2}\lambda_2\bar{\lambda_1}^{-1}\mu,\  {\bf p}_{3}\lambda_3 \lambda_1\mu,\  {\bf p}_{4}\lambda_4 \bar{\lambda_1}^{-1}\mu)\end{equation}  is the desired Gram matrix.
\end{proof}

The matrix  $G$ as in Proposition \ref{normprocess} is called  the {\it normalised Gram matrix}, which is  denoted  by  $G(\mathfrak{p})$.  The  corresponding  lift  ${\bf n}=({\bf n}_{1},{\bf n}_{2},{\bf n}_{3}, {\bf n}_{4})$ given by (\ref{findnormlift}) of $\mathfrak{p}$  is called  the  {\it normalised lift}.  That is
\begin{equation}\label{normmatrix}
G(\mathfrak{p})=G({\bf n})=(g_{ij})=\left(
                             \begin{array}{cccc}
                               0 & 1 & -e^{{\bf i}\ba} & \kappa_1 \\
                               1 & 0 & 1 & \kappa_2 \\
                               -e^{-{\bf i}\ba} & 1 & 0 & 1 \\
 \bar{\kappa_1} & \bar{\kappa_2} & 1 & 0 \\
                             \end{array}
                           \right),
\end{equation}
where $\kappa_1=c_1+t{\bf j}$ and $\kappa_2= c_2+c_3{\bf j}$.

Let $\p=(\p_1,\p_2,\p_3,\p_4)$ be a quadruple of points in $V_0\cup V_-$ such that $\bp(\p)$  are pairwise distinct points in $\overline{{\bf H}_{\bh}^n}$.
We define the following three cross-ratios:
\begin{equation}\label{cross1-2} \X_1({\bf p})=\X({\bf p}_{1},{\bf p}_{2},{\bf p}_{3}, {\bf p}_{4}),\ \X_2({\bf p})=\X({\bf p}_{2},{\bf p}_{4},{\bf p}_{3}, {\bf p}_{1}),\  X_3({\bf p})=\X({\bf p}_{1},{\bf p}_{4},{\bf p}_{3}, {\bf p}_{2}).\end{equation}

\begin{prop}\label{somerelations}
Let $G(\mathfrak{p})$ be a normalised Gram matrix  given by (\ref{normmatrix}) with a normalised lift  $${\bf n}=({\bf n}_{1},{\bf n}_{2},{\bf n}_{3}, {\bf n}_{4})=({\bf p}_{1}\nu_1,\  {\bf p}_{2}\nu_2,\  {\bf p}_{3}\nu_3,\  {\bf p}_{4}\nu_4)$$   given by (\ref{findnormlift}). That is  \begin{equation}\label{findnu12}\nu_1=\lambda_1\mu,\ \nu_2=\lambda_2\bar{\lambda_1}^{-1}\mu ,\ \nu_3=\lambda_3 \lambda_1\mu , \nu_4=\lambda_4 \bar{\lambda_1}^{-1}\mu. \end{equation}
Then
 \begin{itemize}
   \item[(i)] $\X_1({\bf n})=-e^{-{\bf i}\ba} \bar{\kappa_2} \ \bar{\kappa_1}^{-1},\ \X_2({\bf n})=\kappa_1,\ \  \ba_{\bh}({\bf n}_{1},{\bf n}_{2},{\bf n}_{3})=\ba;$

\item[(ii)] $ \kappa_1=\X_2({\bf n}),\  \kappa_2=-\X_2({\bf n})\overline{\X_1({\bf n})}e^{-{\bf i}\ba}, \  \ba=\ba_{\bh}({\bf n}_{1},{\bf n}_{2},{\bf n}_{3});$

\item[(iii)]   $\X_1({\bf p})=\bar{\nu_1}^{-1} \X_1({\bf n})\bar{\nu_1},\  \X_2({\bf p})=\bar{\nu_2}^{-1}\X_2({\bf n})\bar{\nu_2}, \  \ba_{\bh}({\bf p}_{1},{\bf p}_{2},{\bf p}_{3})=\ba_{\bh}({\bf n}_{1},{\bf n}_{2},{\bf n}_{3});$

    \item[(iv)]   $ \kappa_1=\bar{\nu_2}\X_2({\bf p})\bar{\nu_2}^{-1},\  \kappa_2=-\bar{\nu_2}\X_2({\bf p})\bar{\nu_2}^{-1}\nu_1^{-1}\overline{\X_1({\bf p})}\nu_1e^{-{\bf i}\ba_{\bh}({\bf p}_{1},{\bf p}_{2},{\bf p}_{3})}, \  \ba=\ba_{\bh}({\bf p}_{1},{\bf p}_{2},{\bf p}_{3});$

         \item[(v)]  $\X_3({\bf p})=\bar{\nu_1}^{-1} \X_3({\bf n})\bar{\nu_1}$, \ $\X_3({\bf n})=-e^{-{\bf i}\ba} \kappa_2.$

    \end{itemize}
\end{prop}

\begin{proof} Since $\langle \n_1, \n_3\rangle=-e^{{\bf i}\ba} $, $\langle \n_1, \n_4\rangle=\kappa_1$ and $\langle \n_2, \n_4\rangle=\kappa_2$, we have  $$\X_1({\bf n})=\X({\bf n}_{1},{\bf n}_{2},{\bf n}_{3}, {\bf n}_{4})=\langle \n_3, \n_1\rangle \langle \n_3, \n_2
 \rangle^{-1} \langle \n_4, \n_2\rangle \langle \n_4, \n_1\rangle^{-1}=-e^{-{\bf i}\ba} \bar{\kappa_2} \ \bar{\kappa_1},$$
   $$\X_2({\bf n})=\langle \n_3, \n_2\rangle \langle \n_3, \n_4
 \rangle^{-1} \langle \n_1, \n_4\rangle \langle \n_1, \n_2\rangle^{-1}=\kappa_1$$
 and $$\X_3({\bf n})=\langle \n_3, \n_1\rangle \langle \n_3, \n_4
 \rangle^{-1} \langle \n_2, \n_4\rangle \langle \n_2, \n_1\rangle^{-1}=-e^{-{\bf i}\ba} \kappa_2.$$
  By  properties of  quaternionic Cartan's angular  invariant, we have  $\ba_{\bh}({\bf n}_{1},{\bf n}_{2},{\bf n}_{3})=\ba_{\bh}({\bf p}_{1},{\bf p}_{2},{\bf p}_{3})=\ba$.
 Since ${\bf p}_{i}=\n_i\nu_i^{-1}$,  by (\ref{crat1sim}) we have $$\X_1({\bf p})=\X(\n_1\nu_1^{-1},\n_2\nu_2^{-1},\n_3\nu_3^{-1}, \n_4\nu_4^{-1})=\bar{\nu_1}^{-1} \X_1({\bf n})\bar{\nu_1}.$$   Similarly, we obtain  $\X_2({\bf p})=\bar{\nu_2}^{-1}\X_2({\bf n})\bar{\nu_2}$, \ $\X_3({\bf p})=\bar{\nu_1}^{-1}\X_3({\bf n})\bar{\nu_1}.$  Therefore  (i)-(v) hold.
\end{proof}

\begin{prop}\label{negreal}
Let $\kappa_1,\kappa_2,\ba$ be given by (\ref{dfntau}) stemming  from $\p=(\p_1,\p_2,\p_3,\p_4)$. Then
\begin{itemize}
   \item[(i)] $\Re(\kappa_2)=-|\kappa_2|\cos\ba_{\bh}(\p_2,\p_3,\p_4)\leq 0;$
      \item[(ii)] $\cos\ba=\cos\ba_{\bh}({\bf p}_{1},{\bf p}_{2},{\bf p}_{3})\geq 0;$
\item[(iii)] $\Re(\kappa_1 \bar{\kappa_2})=-|\kappa_1||\kappa_2|\cos\ba_{\bh}(\p_1,\p_2,\p_4)\leq 0;$
\item[(iv)] $$|\kappa_1|=|\X_2(\p)|=\frac{ |\langle \p_1, \p_4\rangle | |\langle \p_2, \p_3\rangle |}{|\langle \p_1, \p_2\rangle | |\langle \p_3, \p_4\rangle |},\ |\kappa_2|=|\X_3(\p)|=|\X_1(\p)||\X_2(\p)|=\frac{ |\langle \p_1, \p_3\rangle | |\langle \p_2, \p_4\rangle |}{|\langle \p_1, \p_2\rangle | |\langle \p_3, \p_4\rangle |}.$$
    \end{itemize}
\end{prop}

\begin{proof}
Since $\kappa_2=\langle \n_2, \n_4\rangle=\langle \n_2, \n_3, \n_4\rangle$, by Proposition \ref{prop-1.2} we have $$\Re(-\kappa_2)=\Re(-\langle \n_2, \n_3, \n_4\rangle)=|\kappa_2|\frac{\Re(-\langle \n_2, \n_3, \n_4\rangle)}{|\langle \n_2, \n_3, \n_4\rangle|}=|\kappa_2|\cos\ba_{\bh}(\p_2,\p_3,\p_4)\geq 0.$$
It follows from $e^{{\bf i}\ba}=-\langle \n_1, \n_3\rangle=-\langle \n_1, \n_2, \n_3\rangle$ that $$\cos\ba=\Re(-\langle \n_1, \n_2, \n_3\rangle)=\frac{\Re(-\langle \n_1, \n_2, \n_3\rangle)}{|\langle \n_1, \n_2, \n_3\rangle|}=\cos\ba_{\bh}(\n_1,\n_2,\n_3)=\cos\ba_{\bh}(\p_1,\p_2,\p_3)\geq 0.$$
Noting that $\kappa_1=\langle \n_1, \n_4\rangle$, $\kappa_2=\langle \n_2, \n_4\rangle$ and $\langle \n_2, \n_1\rangle=1$,  we have \begin{eqnarray*}\Re(\kappa_1 \bar{\kappa_2})&=&\Re( \bar{\kappa_2}\kappa_1)=\Re(\langle \n_4, \n_2\rangle \langle \n_1, \n_4\rangle)
=\Re(\langle \n_2, \n_1\rangle \langle \n_4, \n_2\rangle \langle \n_1, \n_4\rangle)\\
&=&-\Re(-\langle {\bf n}_{1},{\bf n}_{2},{\bf n}_{4}  \rangle)=-|\kappa_1||\kappa_2|\cos\ba_{\bh}(\n_1,\n_2,\n_4)=-|\kappa_1||\kappa_2|\cos\ba_{\bh}(\p_1,\p_2,\p_4)\leq 0. \end{eqnarray*}
Therefore we prove (i)-(iii).  The assertion (iv) follows from (ii) and (v) of Proposition \ref{somerelations}.
\end{proof}
We remark that one can find more  relationship of quaternionic cross-ratios in  Sections 3, 4 of \cite{platis}.

\medskip

\subsection{The moduli space}

We know from  Proposition \ref{negreal} that a Hermitian matrix of the form  (\ref{normmatrix}) need satisfy some conditions to be a  normalised Gram matrix.  These
conditions are described in the following theorem.
\begin{thm}\label{homeo}
Let $G$ be a matrix of the  form (\ref{normmatrix}) with \begin{equation}\label{intrcond}\ba\in[0,\pi/2], \ \  \Re(\kappa_1 \bar{\kappa_2})\leq 0,\ \Re(\kappa_2)\leq 0,\   \kappa_1\neq 0,\ \  \kappa_2\neq 0.\end{equation}
Denote by
\begin{equation}\label{defdg} D(G)= 1+|\kappa_1|^2+|\kappa_2|^2-2\Re(\kappa_1)+2\Re(\kappa_2e^{-{\bf i}\ba})+2\Re(\bar{\kappa_1}\kappa_2e^{{\bf i}\ba}). \end{equation}
 Then we have the followings.
 \begin{itemize}
   \item[(i)] When $n=2$, $G$ is the normalised Gram matrix for some quadruple of pairwise distinct points in $\partial {\bf H}_{\bh}^n$  if and only if
\begin{equation}\label{condt1} D(G)=0.\end{equation}
  \item[(ii)] When $n>2$, $G$ is the normalised Gram matrix for some quadruple of pairwise distinct points in $\partial {\bf H}_{\bh}^n$  if and only if
\begin{equation}\label{condt2} D(G)\leq 0.\end{equation}
Moreover,  $D(G)=0$  if and only if  there exist  $\lambda_i\in \bh$ with $\sum_{i=1}^4|\lambda_i|\neq 0$  such that $\sum_{i=1}^4{\bf n}_{i}\lambda_i=0$.
 \end{itemize}
 \end{thm}

\begin{proof}  We first prove necessity.
Suppose that $G({\bf n})$ is the normalised matrix with a normalised lift  ${\bf n}=({\bf n}_{1},{\bf n}_{2},{\bf n}_{3}, {\bf n}_{4})$. Since $G(\mathfrak{p})=G({\bf n})=G(f{\bf n})$, we  have   freedom to choose  some specific  normalised lifts.
Since   ${\rm PSp}(n,1)$ acts doubly transitively on $V_0$, by our normalised process  we may assume that

\begin{equation}\label{findn}{\bf n}_{1}=\left(\begin{array}{ccc}
  0 \\
  \vdots\\
  0\\
  \lambda
\end{array}\right),\ \
{\bf n}_{2}=\left(\begin{array}{ccc}
  \bar{\lambda}^{-1}\\
  0 \\
  \vdots\\
  0
 \end{array}\right),\ \
{\bf n}_{3}=\left(\begin{array}{ccc}
  -\bar{\lambda}^{-1}e^{-{\bf i}\ba}\\
  \alpha \\
  \lambda
 \end{array}\right),\ \
 {\bf n}_{4}=\left(\begin{array}{ccc}
  \bar{\lambda}^{-1} \bar{\kappa_1}\\
  \beta \\
  \lambda \bar{\kappa_2}
 \end{array}\right),
\end{equation}
where $\alpha, \beta$ are column vectors in $\bh^{n-1}$.
By $\langle {\bf n}_{3},{\bf n}_{4}\rangle=1$ and $\langle {\bf n}_{3},{\bf n}_{3}\rangle=\langle {\bf n}_{4},{\bf n}_{4}\rangle=0$, we have
$$\beta^*\alpha=1+\kappa_2e^{-{\bf i}\ba}-\kappa_1,\ |\alpha|^2=2\cos\ba,\ |\beta|^2=-2\Re(\kappa_1 \bar{\kappa_2}).$$
Note that $$|\beta|^2|\alpha|^2 \geq |\beta^*\alpha|^2.$$
That is
\begin{eqnarray*}-4\cos\ba \Re(\kappa_1 \bar{\kappa_2})&\geq& |1+\kappa_2e^{-{\bf i}\ba}-\kappa_1|^2\\
&=&1+|\kappa_1|^2+|\kappa_2|^2-2\Re(\kappa_1)+2\Re(\kappa_2e^{-{\bf i}\ba})-2\Re(\bar{\kappa_2}\kappa_1e^{{\bf i}\ba})\\
&=& D(G)-2\Re(\bar{\kappa_1}\kappa_2e^{{\bf i}\ba})-2\Re(\bar{\kappa_2}\kappa_1e^{{\bf i}\ba})\\
&=& D(G)-2(\bar{\kappa_1}\kappa_2+\bar{\kappa_2}\kappa_1 ) \Re(e^{{\bf i}\ba}).\end{eqnarray*} Therefore
$$D(G)= 1+|\kappa_1|^2+|\kappa_2|^2-2\Re(\kappa_1)+2\Re(\kappa_2e^{-{\bf i}\ba})+2\Re(\bar{\kappa_1}\kappa_2e^{{\bf i}\ba})\leq 0.$$
$D(G)=0$ if and only if $\beta=\alpha \mu$ or $\alpha=\beta \nu$ for some  $\mu,\nu\in \bh$.  Therefore the case  $D(G)=0$ implies that there exist  $\lambda_i\in \bh$ with $\sum_{i=1}^4|\lambda_i|\neq 0$  such that $\sum_{i=1}^4{\bf n}_{i}\lambda_i=0$.   $\alpha$ and $\beta$ are quaternions  for the case $n=2$, which implies that we always have $D(G)=0$ for this case.

For  sufficiency,  we need to find an  ${\bf n}=({\bf n}_{1},{\bf n}_{2},{\bf n}_{3}, {\bf n}_{4})$  whose normalised Gram matrix is $G$ under the conditions (\ref{intrcond}) and (\ref{condt1}),   or (\ref{condt2}).
We consider the normalised polar vectors of the following form:\begin{equation}\label{findn}{\bf n}_{1}=(0,\cdots,0,\lambda)^T, {\bf n}_{2}=(\bar{\lambda}^{-1},0,\cdots,0)^T, {\bf n}_{3}=(-\bar{\lambda}^{-1}e^{-{\bf i}\ba},\alpha^T,\lambda)^T, {\bf n}_{4}=(\bar{\lambda}^{-1} \bar{\kappa_1},\beta^T, \lambda \bar{\kappa_2})^T,\end{equation}
where $\lambda\in \bh, \ \alpha,\ \beta\in \bh^{n-1}$.
We need to find  solutions of the following  underdetermined system of equations:
\begin{eqnarray*}
\beta^*\alpha &=& 1+\kappa_2e^{-{\bf i}\ba}-\kappa_1,\\
 |\alpha|^2&=& 2\cos\ba,\\
  |\beta|^2&=&-2\Re(\kappa_1 \bar{\kappa_2}).
\end{eqnarray*}

We first consider the case $D(G)=0$.

Note that $$D(G)=|1+\kappa_2e^{-{\bf i}\ba}-\kappa_1|^2+4\cos\ba \Re(\kappa_1 \bar{\kappa_2}).$$ If $\ba=\pi/2$  then
$$\lambda=1,\ \alpha=(0,\cdots,0)^T,\  \beta=(\sqrt{-2\Re(\kappa_1 \bar{\kappa_2})},0,\cdots,0)^T.$$
are the desired solutions. If $\ba\neq \pi/2$  then
$$\lambda=1,\ \alpha=(\sqrt{2\cos\ba},0,\cdots,0)^T,\  \beta=\alpha\mu,\ \mbox{where}\  \mu=\frac{1+e^{{\bf i}\ba} \bar{\kappa_2}- \bar{\kappa_1}}{2\cos\ba},$$
are the desired solutions.

For the case  $D(G)<0$  and $n>2$,
$$\lambda=1,\ \alpha=(\sqrt{2\cos\ba},0,\cdots,0)^T,\  \beta=(\frac{1+e^{{\bf i}\ba} \bar{\kappa_2}- \bar{\kappa_1}}{\sqrt{2\cos\ba}},0,\cdots,0,\frac{\sqrt{-D(G)}}{\sqrt{2\cos\ba}})^T.$$
are the desired solutions.
\end{proof}

We mention that  (\ref{intrcond}) and  (\ref{defdg}) can be rephrased in terms of  the parameters $ (c_1,c_2,c_3,t;\ba)$ as
\begin{equation}\label{s-5restrictions-2}\ba\in[0,\pi/2], \ \  \Re(c_1\bar{c_2})+t\Re(c_3)\leq 0,\ \Re(c_2)\leq 0,\ \ t\geq 0,\ \ |c_1|^2+t^2\neq 0,\ \ |c_2|^2+|c_3|^2\neq 0\end{equation}
and
\begin{equation}\label{s5-condtmain-2} D(G)= 1+|c_1|^2+|c_2|^2+|c_3|^2+t^2-2\Re(c_1)+2\Re(c_2e^{-{\bf i}\ba})+2\Re\big((\bar{c_1}c_2+t\bar{c_3})e^{{\bf i}\ba}\big).\end{equation}

\vspace{2mm}

\noindent {\bf Remark 5.1.} \   In the complex case, we can view ${\bf n}=({\bf n}_{1},{\bf n}_{2},{\bf n}_{3}, {\bf n}_{4})$ as a matrix.  Then  the normalised matrix $G$ is  $$G={\bf n}^*J{\bf n}.$$ Therefore $$\det G=|\det {\bf n}|^2 \det J=-|\det {\bf n}|^2\leq 0.$$  $\det G=0$ if and only if $\det {\bf n}=0$.  The case $\det {\bf n}=0$ implies that ${\bf n}_{1},{\bf n}_{2},{\bf n}_{3}$ and  ${\bf n}_{4}$ are linearly  dependent.

\vspace{2mm}

We now prove  Theorem \ref{thm-moduli}.

\medskip

{\it Proof of Theorem \ref{thm-moduli}.}\quad  By our normalised process in Proposition \ref{normprocess},   the map $\tau$ define a map $\tau:\mathcal{M}(n)\to \bm(n)$. Theorem \ref{homeo} implies that such map is bijective. It is obvious that  $\tau:\mathcal{M}(n)\to \bm(n)$ is a homeomorphism  because  $\bm(n)$ has the topology  structure induced from $\bc^3\times \br\times \br$. \hfill$\square$

\medskip
\subsection{The  congruence class of  quadruple of pairwise distinct points  of  $\partial {\bf H}_{\bh}^n$}\label{cong1-quad}

Similarly  to Proposition 2.2 and Corollary 2.3 in \cite{cungus10}, we  can prove  the following proposition.
\begin{prop}\label{sect5-prop}
Let  $\mathfrak{p}=(p_1,p_2,p_3,p_4)$  and  $\mathfrak{q}=(q_1,q_2,q_3,q_4)$  be  two quadruples of pairwise distinct points in $\partial {\bf H}_{\bh}^n$. Then
\begin{itemize}
  \item[(i)] $\mathfrak{p}$ and $\mathfrak{q}$  are congruent in ${\rm PSp}(n,1)$ if and only if their associated Gram matrices are equivalent.
 \item[(ii)] $\mathfrak{p}$ and $\mathfrak{q}$  are congruent in ${\rm PSp}(n,1)$ if and only if $G(\mathfrak{p})=G(\mathfrak{q})$.
\end{itemize}
\end{prop}

We theoretically  solve  the problem  whether  two quadruples of pairwise distinct points in $\partial {\bf H}_{\bh}^n$ are in the same congruence class  or not by  Proposition \ref{sect5-prop} (ii).  We can figure out the corresponding  parameters $(c_1,c_2,c_3,t;\ba)$  by the three steps in Section 1 and compare them.
When one consider this problem only using the geometric invariants  such as  the involved  quaternionic cross-ratios and  quaternionic  Cartan's angular  invariants, Propositions \ref{somerelations} and \ref{negreal} also provide  some equalities to compute the parameters $(c_1,c_2,c_3,t;\ba)$.  However,  we have used $\nu_1$ and $\nu_2$ to describe  the relationship between $\X_i(\p)$ and $\X_i(\n),  i=1,2$  in Proposition \ref{somerelations}  (ii).  Note that  $\nu_1$ and $\nu_2$ are expressed by  the two functions $\nu(a)$ and $\sigma(a,b)$.  Therefore it is difficult for us to obtain the parameters $(c_1,c_2,c_3,t;\ba)$ in terms of $\X_i(\p),  i=1,2,3$ and $\ba_{\bh}(\p_{\iota(1)},\p_{\iota(2)},\p_{\iota(3)}),  \iota\in S_4$.  Here $S_4$ is the symmetric group of degree $4$.

So it makes sense for us to come up with  a theorem about the  congruence classes of quadruples   of pairwise distinct points on $\partial{\bf H}_\bh^n$ only using    geometric invariants.

Before establishing  such a theorem,  we  suggest a way to evade the embarrassing question caused by  $\nu(a)$ and $\sigma(a,b)$. We mention that there exists an interesting  Gram matrix given by
$${\bf m}=({\bf p}_{1}\lambda_1,\  {\bf p}_{2}\lambda_2\lambda_1^{-1},\  {\bf p}_{3}\lambda_3 \lambda_1,\  {\bf p}_{4}\lambda_4 \lambda_1^{-1}),$$
where   \begin{equation*}\label{flam2-4'}\lambda_2=\langle {\bf p}_{2}, {\bf p}_{1}\rangle^{-1},\ \lambda_3=\langle {\bf p}_{3}, {\bf p}_{2}\rangle^{-1}\langle {\bf p}_{1}, {\bf p}_{2}\rangle, \ \lambda_4=\langle {\bf p}_{4}, {\bf p}_{3}\rangle^{-1}\langle {\bf p}_{2}, {\bf p}_{3}\rangle \langle {\bf p}_{2}, {\bf p}_{1}\rangle^{-1} \end{equation*}
 and
 \begin{equation}\label{flam1'}\lambda_1=| \langle {\bf p}_{1}, {\bf p}_{3}\lambda_3\rangle |^{-1/2}=\sqrt{\frac{|\langle {\bf p}_{2}, {\bf p}_{3}\rangle|}{|\langle {\bf p}_{2}, {\bf p}_{1}\rangle  \langle {\bf p}_{1}, {\bf p}_{3}\rangle|}}\ .\end{equation}
 Computation shows that   \begin{equation}\label{inter-gram}G({\bf m} )=(g_{ij})=\left(
                             \begin{array}{cccc}
                               0 & 1 & \omega_3 & \omega_1 \\
                               1 & 0 & 1 & \omega_2 \\
                               \bar{\omega_3} & 1 & 0 & 1 \\
 \bar{\omega_1} & \bar{\omega_2} & 1 & 0 \\
                             \end{array}
                           \right),\end{equation}
 where $$\omega_3=g_{13}=\frac{\langle\p_1,\p_2,\p_3\rangle}{|\langle\p_1,\p_2,\p_3\rangle|}, \   \omega_1=g_{14}=\frac{\langle\p_2,\p_1\rangle \X_2(\p)\langle\p_1,\p_2\rangle}{|\langle\p_1,\p_2\rangle|^2}, \ \omega_2=g_{24}=\omega_3\X_3(\p) .$$

A matrix of  the form  (\ref{inter-gram}) is called  a {\it semi-normalised Gram matrix}.  We represent such a Hermitian matrix by $$V({\bf m})=(\omega_1,\omega_2,\omega_3).$$
Let $G$   and $\tilde{G}$ be two semi-normalised Gram matrix represented by  $(\omega_1, \omega_2, \omega_3)$ and $(\tilde{\omega}_1, \tilde{\omega}_2, \tilde{\omega}_3)$, respectively.  By Proposition \ref{sect5-prop} (i),   $\tilde{G}$ and $G$  are equivalent  if and only if there exists a matrix
 $$D=\mu I_4,\  \mu\in \bh, \ |\mu|=1$$ such that $$\tilde{G} = D^*GD.$$
 That is  $$ (\tilde{\omega}_1, \tilde{\omega}_2, \tilde{\omega}_3)=\bar{\mu}(\omega_1, \omega_2, \omega_3)\mu=(\bar{\mu}\omega_1\mu, \bar{\mu}\omega_2\mu, \bar{\mu}\omega_3\mu). $$
 Based on the above observation, we define an equivalent relation in $\bc\times\bc\times\bc$ by the above equality and denote it by  $$(\tilde{\omega}_1, \tilde{\omega}_2, \tilde{\omega}_3) \simeq (\omega_1, \omega_2, \omega_3).$$
  Let ${\rm M}$ be the set of representations $(\omega_1, \omega_2, \omega_3)$.  Then the configuration space $\mathcal{M}(n)$  can be viewed as the quotient
  of ${\rm M}$ under this equivalent relation.  That is  $\mathcal{M}(n)=M/\simeq$.  The description of ${\rm M}$ is left as an exercise for the  reader.

We revert to the original topic of this subsection.  The following theorem provide a criterion on  classification of congruence class  of two  quadruples of pairwise distinct points in $\partial {\bf H}_{\bh}^n$, which somewhat can be viewed as a description of the above equivalent relation by geometric invariants.

\begin{thm}\label{thmcong}
 Let  $\fp=(p_1, p_2, p_3,p_4)$ and  $\fq=(q_1, q_2, q_3,q_4)$ be two
quadruple of pairwise distinct points in   $\partial {\bf H}_{\bh}^n.$  Then there exists an isometry
$h\in  {\rm Sp}(n,1)$ such that $h(p_1)=q_1, h(p_2)=q_2,h(p_3)=q_3,h(p_4)=q_4$   if and only if  the following conditions hold:
\begin{itemize}
  \item[(i)]$ \X(p_1, p_2, p_3, p_4)\sim \X(q_1, q_2, q_3, q_4)$, $\X(p_1, p_4, p_3, p_2)\sim\X(q_1, q_4, q_3, q_2)$  and $\X(p_2, p_4, p_3, p_1)\sim \X(q_2, q_4, q_3, q_1)$;
  \item[(ii)] $\ba_{\bh}(p_1, p_2, p_3)=\ba_{\bh}(q_1, q_2, q_3)$,\  $\ba_{\bh}(p_1, p_2, p_4)=\ba_{\bh}(q_1, q_2, q_4)$ and $\ba_{\bh}(p_2, p_3, p_4)=\ba_{\bh}(q_2, q_3, q_4)$.
\end{itemize}
\end{thm}

We need the following properties of quaternions to prove Theorem \ref{thmcong}. The proof of them is by direct computations.
\begin{prop}\label{prop-addqu}
\begin{itemize}
  \item[(i)]For any quaternions $a$ and $b$ we have
$$\Re(ab)=\Re(a)\Re(b)+\Re\big(\Im(a)\Im(b)\big);$$
  \item[(ii)] For $a=a_1{\bf i}+a_2{\bf j}+a_3{\bf k}$ and $b=b_1{\bf i}+b_2{\bf j}+b_3{\bf k}$, $a_i,b_i\in \br$,  the angle $\phi$  between vectors $(a_1,a_2,a_3)$ and $(b_1,b_2,b_3)$ is $$\phi=\arccos\big(\frac{\Re(a\bar{b})}{|a b|}\big)=\arccos\big(\frac{-\Re(a b)}{|a b|}\big).$$
\end{itemize}
\end{prop}

 {\it Proof of Theorem \ref{thmcong}.}\quad  The necessity is clear.  Since  ${\rm Sp}(n,1)$ acts transitively on  $\partial {\bf H}_{\bh}^n$  we may assume that $p_1=q_1=o$ and $p_2=q_2=\infty$.  For  sufficiency we need to find  an  $h\in G_{o,\infty}$ such that $$h(p_3)=q_3,\   h(p_4)=q_4$$ under the conditions (i) and (ii).
For convenience, we set $$p_3=(r_1,\cdots,r_n),\  p_4=(s_1,\cdots,s_n),\  q_3=(z_1,\cdots,z_n), \ q_4=(w_1,\cdots,w_n).$$

    The  condition $\X(o,\infty, p_3, p_4)\sim \X(o,\infty, q_3, q_4)$ implies that \begin{equation}\label{0w34}\frac{|r_1|}{|s_1|}=\frac{|z_1|}{|w_1|},\ \ \frac{\Re(r_1\bar{s_1})}{|r_1||s_1|}=\frac{\Re(z_1\bar{w_1})}{|z_1||w_1|}.\end{equation}
    The  condition $\X(o, p_4, p_3,\infty)\sim \X(o, q_4, q_3,\infty)$ implies that \begin{equation}\label{043w}\frac{|r_1|}{|\langle {\hat\p_3},{\hat\p_4}\rangle|}=\frac{|z_1|}{|\langle {\hat\q_3},{\hat\q_4}\rangle|},\ \ \frac{\Re(r_1\overline{\langle {\hat\p_3},{\hat\p_4}\rangle})}{|r_1||\langle {\hat\p_3},{\hat\p_4}\rangle|}=\frac{\Re(z_1\overline{\langle {\hat\q_3},{\hat\q_4}\rangle})}{|z_1||\langle {\hat\q_3},{\hat\q_4}\rangle|}.\end{equation}
    The  condition  $\X(\infty,p_4,p_3, o)\sim \X(\infty,q_4,q_3, o)$ implies that \begin{equation}\label{w430}\frac{|s_1|}{|\langle {\hat\p_3},{\hat\p_4}\rangle|}=\frac{|w_1|}{|\langle {\hat\q_3},{\hat\q_4}\rangle|},\ \ \frac{\Re( s_1\langle{\hat\p_3},{\hat\p_4}\rangle )}{|s_1||\langle {\hat\p_4},{\hat\p_3}\rangle|}=\frac{\Re(w_1\langle {\hat\q_3},{\hat\q_4}\rangle) }{|w_1||\langle {\hat\q_4},{\hat\q_3}\rangle|}.\end{equation}
  The  conditions $\ba_{\bh}(p_1, p_2, p_3)=\ba_{\bh}(q_1, q_2, q_3)$, $\ba_{\bh}(p_1, p_2, p_4)=\ba_{\bh}(q_1, q_2, q_4)$ and
 $\ba_{\bh}(p_2, p_3, p_4)=\ba_{\bh}(q_2, q_3, q_4)$ imply that
\begin{equation}\label{anguc}  \frac{\Re(r_1)}{|r_1|}=\frac{\Re(z_1)}{|z_1|},\ \ \frac{\Re(s_1)}{|s_1|}=\frac{\Re(w_1)}{|w_1|}, \  \frac{\Re(\langle {\hat\p_4},{\hat\p_3}\rangle)}{|\langle {\hat\p_4},{\hat\p_3}\rangle|}=\frac{\Re(\langle {\hat\q_4},{\hat\q_3}\rangle)}{|\langle {\hat\q_4},{\hat\q_3}\rangle|}.\end{equation}
  By Proposition \ref{prop-addqu} and (\ref{0w34})-(\ref{anguc}), we know  that  the three angles between  pairs of the vectors  $$\Im\big(\frac{r_1}{|r_1|}\big), \Im\big(\frac{s_1}{|s_1|}\big), \Im\big(\frac{ \langle {\hat\p_3},{\hat\p_4}\rangle}{ |\langle {\hat\p_3},{\hat\p_4}\rangle|}\big)$$ equal the  three  angles  between  the corresponding  pairs of the vectors $$\Im\big(\frac{z_1}{|z_1|}\big), \Im\big(\frac{w_1}{|w_1|}\big), \Im\big(\frac{ \langle {\hat\q_3},{\hat\q_4}\rangle}{| \langle {\hat\q_3},{\hat\q_4}\rangle|}\big),$$ respectively.

 Let $$\kappa=\frac{|r_1|}{|z_1|}=\frac{|s_1|}{|w_1|}=\frac{|\langle {\hat\p_3},{\hat\p_4}\rangle|}{|\langle {\hat\q_3},{\hat\q_4}\rangle|}.$$  Due to the action of ${\rm Sp}(1)$ by conjugation in $\bh$ coincides with the action of ${\rm SO}(3)$, therefore   there exists a  $\mu\in {\rm Sp}(1)$ such that  $$r_1= \kappa\bar{\mu}z_1\mu, \ \ s_1= \kappa\bar{\mu}w_1\mu,\ \ \langle {\hat\p_3},{\hat\p_4}\rangle= \kappa\bar{\mu}\langle {\hat\q_3},{\hat\q_4}\rangle\mu.$$
Noting that $$\langle {\hat\p_3},{\hat\p_4}\rangle=r_1+\bar{s_1}+\sum_{i=2}^n{\bar{s_i}r_i},\ \langle {\hat\q_3},{\hat\q_4}\rangle=z_1+\bar{w_1}+\sum_{i=2}^n{\bar{w_i}z_i},$$
we have that $$\sum_{i=2}^n{\bar{s_i}r_i}=\kappa\bar{\mu}\big(\sum_{i=2}^n{\bar{w_i}z_i}\big)\mu.$$   Since $p_3, p_4$ and $q_3, q_4$ are on $\partial {\bf H}_{\bh}^n$, we have that
$$\frac{\sum_{i=2}^{n}|r_i|^2}{\sum_{i=2}^{n}|z_i|^2}= \frac{\sum_{i=2}^{n}|s_i|^2}{\sum_{i=2}^{n}|w_i|^2}=\kappa$$
and so that we can find an  $A\in {\rm Sp}(n-1)$ such that
$$A (r_2,\cdots, r_n)^T=\sqrt{\kappa}(z_2,\cdots, z_n)^T\mu,\, A (s_2,\cdots, s_n)^T=\sqrt{\kappa}(w_2,\cdots, w_n)^T\mu.$$
Hence  $h={\rm diag}(\frac{\mu}{\sqrt{\kappa}},A,\sqrt{\kappa}\mu)$ is the desired isometry. \hfill$\square$

\vspace{3mm}
{\bf Acknowledgements}\quad   This work was supported by National Natural Science Foundation of China  and Educational Commission of Guangdong Province.
The authors would like to thank Prof. John R. Parker and  the  referee  for their useful suggestions, which totally reshaped and enhanced  this paper.


\begin{thebibliography}{99}

 \bibitem{apakim07} Apanasov, B. N.,  Kim, I.: Cartan angular  invariant and
deformations of rank 1 symmetric spaces. {\it Sbornik: Mathematics}  {\bf 198}(2),
147-169 (2007)

\bibitem{bea83}Beardon, A. F.:   The geometry of discrete groups, Berlin, New York: Spring-Verlag, (1983)

\bibitem{bisgen09}   Bisi, C., Gentili, G.:  M\"obius transformations and the Poincare distance in the quaternionic setting. {\it Indiana Univ. Math. J.} {\bf 58}, 2729-2764 (2009)

\bibitem{bre90}  Brehm, U.:  The shape invariant of triangles and trigonometry in two-point homogeneous spaces. {\it Geometriae Dedicata.} {\bf 33}, 59-76 (1990)


\bibitem{caowat98} Cao, C., Waterman, P. L.: Conjugacy invariants of M\"obius
groups,  {\it Quasiconformal Mappings and Analysis}. Springer, 109--139 (1998)

\bibitem{caopar11}Cao, W. S.,  Parker, J. R.:  J{\o}rgensen's inequalities and collars in n-dimensional
quaternionic hyperbolic space. {\it  Quarterly J. Math}. {\bf 62}, 523-543 (2011)


\bibitem{car32} Cartan, E.: Sur le groupe de la g$\acute{e}$om$\acute{e}$trie hypersph$\acute{e}$rique. {\it Comment. Math. Helv.} {\bf 4}, 158-171 (1932)

\bibitem{chegre74} Chen, S.S., Greenberg, L.:   Hyperbolic spaces, {\it Contributions
to analysis}.   Academic Press, New York.  49-87 (1974)

\bibitem{cungus10} Cunha, H., Gusevskii, N.: On the moduli space of quadruples of points in the boundary of complex
hyperbolic space. {\it Transform. Groups}  {\bf 15}(2), 261-283 (2010)
\bibitem{cungus12} Cunha, H., Gusevskii, N.:  The moduli space of points in the boundary
of complex hyperbolic space.  {\it J. Geom.  Anal.} {\bf 22}, 1-11 (2012)




\bibitem{falpla08}  Falbel, E., Platis, I. D.:  The ${\rm PU}(2,1)$ confguration space of four points in $S^3$ and the
cross-ratio variety.  {\it Math. Ann.} {\bf 340}(4), 935-962 (2008)

\bibitem{fal09}  Falbel, E.:  A spherical CR structure on the complement of the fgure eight knot with
discrete holonomy. {\it J. Diferential Geom.} {\bf 79}(1), 69-110 (2008)


\bibitem{gol99} Goldman, W. M.:  Complex hyperbolic geometry.  In: Oxford Mathematical Monographs. Oxford Science
Publications. The Clarendon Press, Oxford University Press, New York (1999)

\bibitem{gro06} Grossi,  C.:   PhD Thesis, Universidade Estadual de Campinas, (2006)


\bibitem{kimpar03} Kim, I., Parker, J. R.: Geometry of quaternionic hyperbolic
manifolds.  {\it Math. Proc. Cambridge Philos. Soc.} {\bf 135}, 291-320
(2003)


\bibitem{korrei87} Kor\'anyi,  A.,  Reimann, H. M.:   The complex cross-ratio on the
Heisenberg group.  {\it Enseign. Math.} {\bf 33}, 291-300 (1987)

\bibitem{park}Parker, J. R.:  Notes on complex hyperbolic geometry. (2010).

\bibitem{park18}  Parker,  J. R., Platis, I. D.:  Complex hyperbolic Fenchel-Nielsen coordinates. {\it Topology}
{\bf 47}(2), 101-135 (2008)

\bibitem{park19} Parker,  J. R., Platis, I. D.:  Global geometrical coordinates on Falbel's cross-ratio variety.
{\it Canad. Math. Bull.} {\bf 52}, 285-294  (2009)


\bibitem{platis} Platis, I. D.:  Cross-ratios and the Ptolemaean inequality in boundaries of symmetric spaces of rank 1. {\it Geometriae Dedicata.} {\bf 169}, 187-208 (2014)


\end{thebibliography}
\end{document}